\documentclass[12pt]{article}
\usepackage{amssymb,amsfonts,epic
} 
\begin{document}

\newtheorem{lem}{Lemma}[section]
\newtheorem{prop}{Proposition}
\newtheorem{con}{Construction}[section]
\newtheorem{defi}{Definition}[section]
\newtheorem{coro}{Corollary}[section]
\newcommand{\hf}{\hat{f}}
\newtheorem{fact}{Fact}[section]
\newtheorem{theo}{Theorem}
\newcommand{\Br}{\Poin}
\newcommand{\Cr}{{\bf Cr}}
\newcommand{\dist}{{\bf dist}}
\newcommand{\diam}{\mbox{diam}\, }
\newcommand{\mod}{{\rm mod}\,}
\newcommand{\compose}{\circ}
\newcommand{\dbar}{\bar{\partial}}
\newcommand{\Def}[1]{{{\em #1}}}
\newcommand{\dx}[1]{\frac{\partial #1}{\partial x}}
\newcommand{\dy}[1]{\frac{\partial #1}{\partial y}}
\newcommand{\Res}[2]{{#1}\raisebox{-.4ex}{$\left|\,_{#2}\right.$}}
\newcommand{\sgn}{{\rm sgn}}

\newcommand{\C}{{\bf C}}
\newcommand{\D}{{\bf D}}
\newcommand{\Dm}{{\bf D_-}}
\newcommand{\N}{{\bf N}}
\newcommand{\R}{{\bf R}}
\newcommand{\Z}{{\bf Z}}
\newcommand{\tr}{\mbox{Tr}\,}

\newenvironment{nproof}[1]{\trivlist\item[\hskip \labelsep{\bf Proof{#1}.}]}
{\begin{flushright} $\square$\end{flushright}\endtrivlist}
\newenvironment{proof}{\begin{nproof}{}}{\end{nproof}}

\newenvironment{block}[1]{\trivlist\item[\hskip \labelsep{{#1}.}]}{\endtrivlist}
\newenvironment{definition}{\begin{block}{\bf Definition}}{\end{block}}

\newtheorem{conjec}{Conjecture}
\newtheorem{com}{Comment}
\newtheorem{exa}{Example}
\font\mathfonta=msam10 at 11pt
\font\mathfontb=msbm10 at 11pt
\def\Bbb#1{\mbox{\mathfontb #1}}
\def\lesssim{\mbox{\mathfonta.}}
\def\suppset{\mbox{\mathfonta{c}}}
\def\subbset{\mbox{\mathfonta{b}}}
\def\grtsim{\mbox{\mathfonta\&}}
\def\gtrsim{\mbox{\mathfonta\&}}

\newcommand{\ar}{{\bf area}}
\newcommand{\1}{{\bf 1}}
\newcommand{\Bo}{\Box^{n}_{i}}
\newcommand{\Di}{{\cal D}}
\newcommand{\gd}{{\underline \gamma}}
\newcommand{\gu}{{\underline g }}
\newcommand{\ce}{\mbox{III}}
\newcommand{\be}{\mbox{II}}
\newcommand{\F}{\cal{F_{\eta}}}
\newcommand{\Ci}{\bf{C}}
\newcommand{\ai}{\mbox{I}}
\newcommand{\dupap}{\partial^{+}}
\newcommand{\dm}{\partial^{-}}
\newenvironment{note}{\begin{sc}{\bf Note}}{\end{sc}}
\newenvironment{notes}{\begin{sc}{\bf Notes}\ \par\begin{enumerate}}%
{\end{enumerate}\end{sc}}
\newenvironment{sol}
{{\bf Solution:}\newline}{\begin{flushright}
{\bf QED}\end{flushright}}

\title{Perturbations of weakly expanding critical orbits}
\date{}
\author{Genadi Levin
\thanks{Supported in part by ISF grant 799/08}\\
\small{Institute of Mathematics, the Hebrew University of Jerusalem, Israel}\\
} 
\normalsize 
\maketitle 
\abstract{Let $f$ be a polynomial or a rational function which
has $r$ summable critical points. We prove that there exists
an $r$-dimensional manifold $\Lambda$ in an appropriate space containing $f$
such that for every smooth curve in $\Lambda$ through $f$,
the ratio between parameter and dynamical derivatives along 
forward iterates of at least
one of these summable points tends to a non-zero number.} \

\section{Introduction}
We say that a critical point $c$ of a rational function $f$ is 
{\it weakly expanding, or summable}, if, for the point
$v=f(c)$ of the Riemann
sphere ${\bf \bar C}$,
\begin{equation}\label{summ}
\sum_{n=0}^\infty \frac{1+|f^n(v)|^2}{1+|v|^2}
\frac{1}{|(f^n)'(v)|}<\infty.
\end{equation}
Throughout the paper, derivatives are standard derivatives of 
holomorphic maps; then the summand in~(\ref{summ}) is
a finite number for every $v\in {\bf \bar C}$ as soon as
$v$ is not a critical point of $f^n$.

In the present paper, we study perturbations of polynomials and
rational functions with several (possibly, not all)
summable critical points. 
The paper
is a natural continuation of~\cite{Le},~\cite{mult}, and partly~\cite{lpade},
~\cite{LSY}.  
Let us state main result for rational functions
(for a more complete account, see Theorem~\ref{onerat}).
We call a rational function {\it exceptional}, if it is 
{\it double covered by an integral torus
endomorphism} (another name: {\it a flexible Lattes map}): 
these form a family of explicitly described 
critically finite rational maps with Julia sets the Riemann sphere,
see e.g.~\cite{dht},~\cite{mcm},~\cite{mill}.
Two rational functions are {\it equivalent}, if they
are conjugated by a Mobius transformation.
\begin{theo}\label{tworatintro}
Let $f$ be an arbitrary non-exceptional rational function of degree $d\ge 2$.
Suppose that $c_1,...,c_r$ is a collection 
of $r$ summable critical points of $f$,
and the union $K=\cup_{j=1}^r \omega(c_j)$ of their $\omega$-limit sets 
is a $C$-compact on the Riemann sphere
(see Definition~\ref{ccomdef}). For example, it
is enough that
$K$ has zero Lebesgue measure on the plane. 
Replacing if necessary $f$ by its
equivalent, 
one can assume the forward orbits
of $c_1,...,c_r$ avoid infinity.
Consider the set $X_f$ of all rational functions of degree $d$
which are close enough to $f$ and have the same number $p'$
of different critical points with the same corresponding multiplicities.
Then there is a $p'$-dimensional manifold $\Lambda_f$ and its
$r$-dimensional submanifold $\Lambda$, $f\in \Lambda\subset \Lambda_f
\subset X_f$, with the following properties: 

(a) every $g\in X_f$ is equivalent
to some $\hat g\in \Lambda_f$,

(b) for every one-dimensional family $f_t\in \Lambda$ through $f$, 
such that $f_t(z)=f(z)+t u(z)+O(|t|^2)$ as $t\to 0$, 
if $u\not=0$, then,
for some $1\le j\le r$, the limit
\begin{equation}\label{familyratintro}
\lim_{m\to \infty}\frac{\frac{d}{dt}|_{t=0}f_t^m (c_j(t))}
{(f^{m-1})'(v_j)}=\sum_{n=0}^\infty \frac{u(f^n(c_j))}{(f^{n})'(f(c_j))}
\end{equation}
exists and is a {\it non-zero} number.
Here $c_j(t)$ is the critical point of $f_t$, such that $c_j(0)=c_j$,
and $v_j=f(c_j)$.
Furthermore, if $f$ and all the critical points of $f$ are real, the above
maps and spaces can be taken real.
\end{theo}
If $f_t(z)=z^d+t$, the limit above is a
{\it similarity factor} between the dynamical and the parameter planes,
see~\cite{R-L0}.

Let us list some cases when the set $K$ defined above has zero 
Lebesgue measure.

({\bf 1}). Suppose that every 
critical point $c$ of $f$ satisfies the Misiurewicz condition,
i.e., $c$ lies in the Julia set $J$ of $f$ and
$\omega(c)$ contains no critical points 
and parabolic cycles. Then $\omega(c)$ is an invariant hyperbolic set of $f$
(by Mane's Hyperbolicity Theorem~\cite{ma}).
By the bounded distortion property, it follows that $\omega(c)$
is of measure zero. 
Therefore, if $c_1,...,c_{p'}$  satisfy the Misiurewicz condition,
the Lebesgue measure $|K|$ of $K$ is zero.

({\bf 2}). If all 
critical points  
satisfy the Collet-Eckmann condition
(which clearly implies the summability), 
and the 
$\omega$-limit set of each of them is not the whole sphere, then
the measure of their union is zero~\cite{PR}. See also~\cite{PR1}
for a rigidity result for Collet-Eckmann holomorphic maps.

({\bf 3}). If all the 
critical points in the Julia set $J$ of a rational function $f$
are summable, $f$ has no neutral cycles, 
and $J$ is not the whole sphere, then
the Lebesgue measure of $J$ is equal to zero
~\cite{BS},~\cite{R-L}. In particular, $|K|=0$.
Moreover, as in the case (2), if $J={\bf \bar C}$, 
it is enough to assume that the 
$\omega$-limit set of each of them is not the whole sphere, 
and then again  $|K|=0$
(the proof follows from~\cite{RS}, as explained in~\cite{R-L1}).

Theorem~\ref{tworatintro} being applied in the case {\bf (1)}
of the list above yields the following corollary first obtained in~\cite{S}.
Let $f$ be a non-exceptional rational function of degree $d$.
Assume that every critical point $c_j$ of $f$, $1\le j\le 2d-2$, is simple
and satisfies the Misiurewicz condition. 
Replacing if necessary $f$ by its
equivalent, 
one can assume that the set $K=\cup_{j=1}^{2d-2}\omega(c_j)$ 
does not contain infinity.
Let $f_\lambda$ be
a family of rational maps of degree $d$
which depends (complex) analytically on $\lambda\in N$,
where $N$ is a $2d-2$-dimensional complex manifold,
and $f_{\lambda_0}=f$. 
In particular, there exist analytic functions
$c_j(\lambda)$, $1\le j\le 2d-2$, such that $c_j(\lambda)$ are 
the critical points of $f_\lambda$,  $c_j(\lambda_0)=c_j$.
Furthermore, 
since the set $K\subset {\bf C}$ is 
hyperbolic for $f$, 
it persists for $f_\lambda$, for $\lambda$ in some neighborhood $W$ of
$\lambda_0$:
there exists a continuous map $(\lambda, z)\in W\times K
\mapsto z(\lambda)\in {\bf C}$
(here $z(\lambda)$ is the ''holomorphic motion'' 
of the point $z\in K$), so that
$z=z(\lambda_0)\in K$, the functions $z(\lambda)$
are analytic in $W$,
and, moreover, for each $\lambda\in W$,
the points $z(\lambda)$ are different for different initial points
$z=z(\lambda_0)$, and the set
$K(\lambda)=\{z(\lambda)\}$ is an 
invariant hyperbolic set for $f_\lambda$.
\begin{coro}\label{misiu} (\cite{S})
Assume the preceding hypotheses on the map $f$ and the family $f_\lambda$
and also that no two different maps from $N$ 
are equivalent. 
Let $a_j=f^l(c_j)\in K$, where $l>0$, and
$a_j(\lambda)$ be the holomorphic motion of the point $a_j$.
Then, for every $l$ big enough, the map 
\begin{equation}
{\bf h}_l: \lambda\mapsto \left\{f^l_\lambda(c_1(\lambda))-a_1(\lambda),...,
f^l_\lambda(c_{2d-2}(\lambda))-a_{2d-2}(\lambda)\right\}
\end{equation}
is (locally) invertible near $\lambda_0$.
\end{coro}
See Subsection~\ref{mainres} for the details and for an asymptotics 
of the inverse to the derivative of ${\bf h}_l$.

Let us make some further comments.
In~\cite{ts}, Theorem~\ref{tworatintro} was shown for real quadratic 
polynomials which satisfy the Collet-Eckmann condition 
(to be precise, Tsujii proves that the 
limit~(\ref{familyratintro}) is positive for real Collet-Eckmann maps
$t-x^2$).
For all quadratic polynomials with the summable critical point,
Theorem~\ref{tworatintro} was originally proven in~\cite{Le}
(see also~\cite{Av}).
In turn, Corollary~\ref{misiu} had been known before
for the quadratic polynomial family, in the case of 
a strictly preperiodic critical point~\cite{dh}. 
See also
~\cite{R-L0} for Corollary~\ref{misiu} for the family $z^d+c$, and
~\cite{be} for a proof 
in the particular case of Corollary~\ref{misiu}
when each critical point is strictly preperiodic.

Here we state another corollary of Theorem~\ref{tworatintro},
which is close to the main result of~\cite{mak}, see Comment~\ref{tha}
below.
Recall the following terminology.
Let $X$ be a complex manifold, and $f_\lambda (z)$, $\lambda\in X$, 
be a holomorphic family of rational maps over $X$.
The map $f_{\lambda_0}$ is called structurally 
(respectively, quasiconformally) stable in $X$,
if there exists a neighborhood $U$ of $\lambda_0$, such that:
(a) $f_{\lambda_0}$ and $f_{\lambda}$ are topologically 
(respectively, quasiconformally) conjugated,
for every $\lambda\in U$, (b) the conjugacy tends 
to the identity map as $\lambda\to \lambda_0$.
By a fundamental result of~\cite{MSS},~\cite{mcmsul},
the sets of structurally stable and quasiconformally stable rational maps
in $X$ coincide, and, moreover, the quasiconformal conjugacies form
a holomorphic motion. In particular, the condition (b) of the definition
follows from (a).
(We will not use these results though.) 
In the next statement, $f$ is not necessarily non-exceptional.
Note that the set of functions $X_f$ and the manifold $\Lambda_f$ appeared in
Theorem~\ref{tworatintro} are defined without any changes for all maps $f$
(including the exceptional ones), see 
Subsection~\ref{disc} and Proposition~\ref{gen} in particular.
\begin{coro}\label{cormak} (cf.~\cite{mak})
Suppose that a rational function $f$ of degree $d\ge 2$ 
has a summable critical point $c$,
and its $\omega$-limit set $\omega(c)$ is a $C$-compact.
Then $f$ is not structurally stable in $\Lambda_f$.
\end{coro}
For the proof, see Subsection~\ref{mainres}.
\begin{com}\label{tha}
In this comment, we 
consider the case when $f$ has only simple critical points. 
Then (and only then)  
$X_f$ is the (local)
space of all rational functions ${\bf Rat}_d$ of degree $d$ which are
close to $f$, and $\Lambda_f$ is its $2d-2$-dimensional submanifold. 
In this case,
Corollary~\ref{cormak} was proved essentially in~\cite{mak} (Theorem A).
To be more precise, in Theorem A of~\cite{mak}, explicit conditions (1)-(4) 
are given, so that
each of them implies that $f$ is not quasiconformally stable in 
the space ${\bf Rat}_d$. 
In fact, each of these 
conditions (2)-(4) (but not (1)) implies that $\omega(c)$ is a $C$-compact;
this is the only fact needed to show that
$f$ is not quasiconformally stable
in this part
of the proof from~\cite{mak}. 
On the other hand, the condition (1) of Theorem A
(which says that $c\notin\omega(c)$) requires a bit different considerations 
and is not covered by the above Corollary~\ref{cormak}.
\end{com}

For other results about
the transversality, see~\cite{mult},
preprint~\cite{e1}, and also~\cite{Gau}. 
For dynamical and statistical properties of one-dimensional and 
rational maps under different summability conditions,
see~\cite{NS},~\cite{GS}, 
~\cite{BS},~\cite{R-L},~\cite{RS}.

The main results of the paper are contained in Theorem~\ref{one}
(plus Comment~\ref{onecoro}) for polynomials,
and in Theorem~\ref{onerat} and Theorem~\ref{tworatintro}
(stated above) for rational functions. 
See also Comment~\ref{last} for
a generalization which takes into account 
non-repelling cycles, 
and Propositions~\ref{main} and~\ref{mainrat}, 
which are of an independent interest.
As usual, the polynomial case is more transparent and
technically easier, so we consider it separately, see Section~\ref{poly}.
Proposition~\ref{dergenlem} and Lemma~\ref{formula} of this Section
are more general and used also in the next
Section~\ref{rat}, where, by the same method, we treat the case of
rational functions.

One of the main tools of the paper is a Ruelle (or pushforward)
operator $T_f$ associated to a rational map $f$ 
(see Subsection~\ref{sectruelle}). It is introduced to the field
of complex dynamics by Thurston, see~\cite{dht},
and is used widely since then. A fundamental property of the operator
(observed by Thurston) is that it is contracting (see 
Subsection~\ref{wd}).
The method of the proof of
Theorem~\ref{one} and Theorem~\ref{onerat}
consists in applying the 
following three basic components: {\bf (i)} 
explicit (formal) identities involving
the Ruelle operator, which are established in~\cite{LSY}
(as a formula for the resolvent of $T_f$ 
and in the case when $f$ has simple critical points),
see also~\cite{lpade}; this allows us to construct explicitely
an integrable fixed point of the operator $T_f$
assuming that the determinant of some matrix
vanishes (in other words, if some vectors of the coefficients
in the identities are linearly dependent),
{\bf (ii)} the contraction property of the operator together with {\bf (i)} 
lead to the conclusion that the vectors of the coefficients
are linearly independent.
The scheme {\bf (i)}-{\bf (ii)} appeared in~\cite{Le},
where the unicritical case was treated, and then applied
in~\cite{mak} for rational functions with simple critical points, 
and in~\cite{mult}.

The last component, which is the key new part of the paper, is
as follows: {\bf (iii)} a formula for the
coefficients of the identities for an arbitrary polynomial
or rational function via the derivative 
with respect to the canonical local coordinates
in some functional spaces,
see Propositions~\ref{main}-\ref{vaznoe} 
and~\ref{mainrat}-\ref{vaznoerat}.

\begin{com}\label{nota}
We will consider different subsets $N$ 
of polynomials or rational functions equipped by the
structure of a complex-analytic $l$-dimensional
manifold and always having the following property: nearby maps in $N$ have
the same number of distinct critical points with the same corresponding 
multiplicities. In other words,
given $f\in N$ with the distinct critical points $c_1,...,c_{q'}$
and corresponding multiplicities $m_1,...,m_{q'}$, there are
$q'$ functions $c_1(g),...,c_{q'}(g)$ which are defined
for $g$ in a neighborhood of $f$ in $N$ such that
$c_j(g)\to c_j$ as $g\to f$ and 
$c_j(g)$ is a critical point of $g$ with multiplicity $m_j$.
It is easy to see that then each
$c_j(g)$ is 
a holomorphic function of $g\in N$ (provided $c_j\in {\bf C}$).
\end{com}
{\it Throughout the paper we 
use the following convention about the notations.}
Let $f\in N$, and $\bar x(g)=\{x_1(g),...,x_l(g)\}$ be 
a (local, near $f$) holomorphic coordinate of $g$ in the space $N$. 
If $P: W\to {\bf C}$ is a function which is defined and analytic 
in a neighborhood $W$ of $f$ in the space $N$, we denote by 
$\frac{\partial P}{\partial x_k}$ the partial derivative
of $P(g)$ w.r.t. $x_k(g)$ calculated at the point $\bar x(f)$, 
i.e., at $g=f$.
For example, $\frac{\partial c_j}{\partial x_k}$
denotes $\frac{\partial c_j(g)}{\partial x_k}$ evaluated at $g=f$. 
Furthermore, if $P=g^m$, 
we denote $\frac{\partial f^m}{\partial x_k}(z)$ to be 
$\frac{\partial g^m}{\partial x_k}(z)$ evaluated at $g=f$.
For a rational function $g(z)$, $g'$ always means the derivative 
w.r.t. $z\in {\bf C}$. For instance,
$\frac{\partial f'}{\partial v_k}$
means $\frac{\partial (\partial g/\partial z)}{\partial v_k}$
calculated at the point $f$.
Note also that for a critical point $c(g)=c_j(g)$ of $g$, we have:
$\frac{\partial (g^m(c(g)))}{\partial x_k}|_{g=f}=
\frac{\partial f^m}{\partial x_k}(c)$, where $c=c(f)$.

\

{\bf Acknowledgments}. 
The paper was inspired by a recent question by Weixiao Shen
to the author about a generalization of Corollary 1 (b) 
of~\cite{Le} to higher degree polynomials.
The answer is contained in Theorem~\ref{one}.
In turn, it has been used recently in~\cite{gash}.
The author thanks Weixiao Shen
for the above question and discussions, 
Feliks Przytycki for discussions, and Juan Rivera-Letelier for few
very helpful comments and for the reference~\cite{Gau}.
Finally, the author thanks the referee for many comments
that helped to improve the exposition.

\section{Polynomials}\label{poly}
\subsection{Polynomial spaces}
Let $f$ be a monic centered 
polynomial of degree $d\ge 2$, i.e., it has the form
$$f(z)=z^d+a_1 z^{d-2}+...+a_{d-1}.$$
Consider the space $\Pi_d$ of all 
monic and centered polynomials of the same degree $d$.
Vector of coefficients of $g\in \Pi_d$
defines a (global) coordinate in $\Pi_d$ and identifies
$\Pi_d$ with ${\bf C}^{d-1}$. 

Let 
$C=\{c_1,...,c_p\}$ be the set of all {\it different} critical points of $f$,
with the vector 
of multiplicities
$\bar p=\{m_1, m_2,...,m_p\}$,
so that $f'(z)=d\Pi_{i=1}^p (z-c_i)^{m_i}$..
Now, we consider a local subspace $\Pi_{d,\bar p}$ of $\Pi_d$
associated to the polynomial $f$, as in~\cite{mult}:
\begin{defi}\label{poldef}
The space $\Pi_{d,\bar p}$ consists of all $g$ 
from a neighborhood of $f$ in $\Pi_d$
with the same number $p$
of different critical points $c_1(g),...,c_p(g)$
and the same vector of multiplicities $\bar p$.
Here $c_i(g)$ is close to $c_i=c_i(f)$ and has the multiplicity
$m_i$, $1\le i\le p$. 
\end{defi}
In particular, $\Pi_{d, \overline{d-1}}$ 
consists of all monic centered polynomials
of degree $d$ close to $f$
if and only if
all the critical points of $f$ are simple. At the other extreme case, the
space $\Pi_{d, \bar 1}$ consists of the unicritical family
$z^d+v$.

Note that some of the critical values $v_1,...,v_p$ of $f$ may coincide.

Consider the vector of critical values
$V(g)=\{v_1(g),...,v_p(g)\}$, where $v_i(g)=g(c_i(g))$.
The set $\Pi_{d, \bar p}$ is an analytic subset of $\Pi_d$.
The following fact is proved in Proposition 1 of~\cite{mult}:
\begin{prop}\label{local}
$\Pi_{d, \bar p}$ is a $p$-dimensional complex
analytic manifold, and the vector $V(g)$ is a local analytic coordinate
in $\Pi_{d, \bar p}$.
\end{prop}

\subsection{Main result}
\begin{theo}\label{one}

(a) Let $c$ be a weakly expanding critical point 
of $f$ and $v=f(c)$. Then,
for every $k=1,...,p$,
the following limit exists:
\begin{equation}\label{lim}
L(c, v_k):=\lim_{m\to \infty}
\frac{\frac{\partial (f^m (c))}{\partial v_k}}{(f^{m-1})'(v)}.
\end{equation}

(b) Suppose that $c_1,...,c_r$ are pairwise different weakly expanding
critical points of $f$.
Then the rank of the matrix
\begin{equation}\label{matrix}
{\bf L}=(L(c_j, v_k))_{1\le j\le r, 1\le k\le p}
\end{equation} 
is equal to $r$, i.e., maximal.
\end{theo} 
\begin{com}\label{onecoro}
Part (b) has the following geometric re-formulation.
Let 
$1\le k_1<...<k_r\le p$ be indexes, for which
determinant of the square matrix 
$(L(c_j, v_{k_i}))_{1\le j\le r, 1\le i\le r}$ is non-zero.
We define a local $r$-dimensional submanifold $\Lambda$
as the set of all $g\in \Pi_{d, \bar p}$ such that
$v_i(g)=0$ for every $i\not=k_1,k_2,....,k_r$.
Consider a $1$-dimensional
family (curve) of maps $f_t\in \Lambda$ through $f$, 
such that $f_t(z)=f(z)+t u(z)+O(|t|^2)$ as $t\to 0$. 
If $u\not=0$, then,
for at least one weakly expanding critical point $c_j$ of $f$,
\begin{equation}\label{family}
\lim_{m\to \infty}\frac{\frac{d}{dt}|_{t=0}f_t^m (c_j(t))}
{(f^{m-1})'(v_j)}\not=0.
\end{equation}
Here $c_j(t)$ is 
a holomorphic function with $c_j(0)=c_j$, so that $c_j(t)$
is a critical point of $f_t$,
and $v_j=f(c_j)$.

In particular, if all critical points of $f$ are {\it simple}
and {\it weakly expanding},
then (\ref{family}) holds for every curve $f_t$
in $\Pi_d$ through $f$ with a non-degenerate tangent vector at $f$.

Furthermore, if $f(0)$ and all the critical points of $f$ are real, the above
maps and spaces can be taken real.
\end{com}
Indeed, let $v_k(t)$ be 
the critical value of $f_t$, such that
$v_k(0)=v_k$. Denote $a_k=v'_k(0)$.
As $f_t\in \Lambda$,
$u(z)=\sum_{i=1}^r a_{k_i} \frac{\partial f}{\partial v_{k_i}}(z)$,
where at least one of $a_{k_i}$ must be non zero. On the other hand, 
the limit $R_j$ in~(\ref{family}) can be represented as
$R_j=\sum_{k=1}^p a_k L(c_j, v_k)$.
Therefore, $R_j=\sum_{i=1}^r a_{k_i} L(c_j, v_{k_i})$, 
where at least one of $a_{k_i}$ is not zero.
Now, if we assume that $R_j=0$ for every $1\le j\le r$, then the matrix 
$(L(c_j, v_{k_i}))_{1\le j\le r, 1\le i\le r}$ degenerates, which is a 
contradiction.
\subsection{Proof of Part (a) of Theorem~\ref{one}}\label{pra}
\paragraph{Dependence on the local coordinates.}

Suppose that a polynomial or a rational function $f$
is included
in a space $N$ of polynomials or rational functions $g$
with coordinates 
$\bar x(g)=\{x_1(g),...,x_l(g)\}$, see Comment~\ref{nota}.
Fix a critical point $c\in {\bf C}$ 
of $f$ of multiplicity $m\ge 1$, and 
let $c(g)$ be the critical point of $g$ of the multiplicity $m$ 
which is close to $c$, if $g$ is close to $f$.
Consider also the corresponding critical value
$v(g)=g(c(g))$. Recall that, according to our convention,
$\frac{\partial f}{\partial x_k}$, $\frac{\partial f(a)}{\partial x_k}$,
$\frac{\partial v}{\partial x_k}$, etc.
mean respectively 
$\frac{\partial g}{\partial x_k}|_{g=f}$, 
$\frac{\partial g(a)}{\partial x_k}|_{g=f}$,
$\frac{\partial v(g)}{\partial x_k}|_{g=f}$, etc.
\begin{prop}\label{dergenlem}
Assume that $c$ and $v=f(c)$ lie in ${\bf C}$.
Then the function
\begin{equation}
\frac{\partial f}{\partial x_k}(z)-\frac{\partial v}{\partial x_k}
\end{equation}
has at $z=c$ a 
zero of multiplicity
at least $m$. 
\end{prop}
\begin{proof}
Fix a point $a\in {\bf C}$ so that $f(a)\not=\infty$.
For every $g$ in the space $N$ which is close to $f$,
we may write:
\begin{equation}\label{1rat}
g(z)=g(a)+\int_{a}^z g'(w) dw,
\end{equation}
and this holds for every $z$ in the plane with $g(z)\not=\infty$.
Hence, 
\begin{equation}\label{2rat}
\frac{\partial f}{\partial x_k}(z)=\frac{\partial f(a)}{\partial x_k}+
\int_{a}^z \frac{\partial f'}{\partial x_k}(w) dw.
\end{equation}
On the other hand,
\begin{equation}\label{3rat}
v(g)=g(a)+\int_{a}^{c(g)} g'(w) dw.
\end{equation}
Therefore,
\begin{equation}\label{4rat}
\frac{\partial v}{\partial x_k}=\frac{\partial f(a)}{\partial x_k}+
\frac{\partial c}{\partial x_k} f'(c)+
\int_{a}^{c} \frac{\partial f'}{\partial x_k}(w) dw=
\frac{\partial f(a)}{\partial x_k}+
\int_{a}^{c} \frac{\partial f'}{\partial x_k}(w) dw.
\end{equation}
Comparing this with~(\ref{2rat}), we have:
\begin{equation}\label{24rat}
\frac{\partial f}{\partial x_k}(z)-\frac{\partial v}{\partial x_k}
=
\int_{c}^z \frac{\partial f'}{\partial x_k}(w) dw.
\end{equation}
As $c(g)$ is an $m$-multiple root of $g'$,
we get:
$$\frac{\partial f'}{\partial x_k}=(z-c)^{m-1}r(z),$$
where $r(z)$ is a holomorphic function near $c$.
Hence, as $z\to c$,
\begin{equation}\label{24'rat}
\frac{\partial f}{\partial x_k}(z)-\frac{\partial v}{\partial x_k}
=(z-c)^{m}\frac{r(c)}{m}+O(z-c)^{m+1}.
\end{equation}
This proves the statement.
\end{proof}
Now, we let $N=\Pi_{d, \bar p}$ and calculate the partial derivatives 
of a function $f\in\Pi_{d, \bar p}$ w.r.t. the local coordinates.
\begin{prop}\label{derlem}
For every $k=1,...,p$, the function
$\frac{\partial f}{\partial v_k}(z)$ is a 
polynomial $p_k(z)$
of degree at most $d-2$ which is uniquely characterized by the
following condition: $p_k(z)-1$ has zero at $c_k$ of order
at least $m_k$, while for every $j\not=k$,
$p_k(z)$ has zero at $c_j$ of order at least $m_j$.
In particular, $\frac{\partial f}{\partial v_k}(c_j)=\delta_{j,k}$
(here and later on we use the notation
$\delta_{j,k}=1$ if $j=k$ and $\delta_{j,k}=0$ if $j\not=k$);
if $c_k$ is simple (i.e., $m_k=1$),
then
\begin{equation}\label{dersimp}
\frac{\partial f}{\partial v_k}(z)=\frac{f'(z)}{f''(c_k)(z-c_k)}.
\end{equation}
\end{prop}
\begin{proof}
Since the coefficients of $g\in \Pi_{d, \bar p}$
are holomorphic functions of $V(g)$ and $g$ is centered, 
the function $\frac{\partial f}{\partial v_k}(z)$ is indeed a polynomial
in $z$ of degree at most $d-2$. 
Hence, it is enough to check that it satisfies the characteristic property
of the polynomial $p_k(z)$. 
But since $\frac{\partial v_j(g)}{\partial v_k}=\delta_{j,k}$,
this is a direct corollary of Proposition~\ref{dergenlem}
applied for the coordinate $\bar x=V$ and the critical point 
$c(g)=c_j(g)$.
\end{proof}

\paragraph{Proof of Theorem~\ref{one}, Part (a).}
The following identity is easy to verify:
\begin{equation}\label{5}
\frac{\partial f^m}{\partial v_k}(z)=
(f^m)'(z)\sum_{n=0}^{m-1}\frac{\frac{\partial f}{\partial v_k}(f^n(z))}
{(f^{n+1})'(z)}.
\end{equation}
Letting here $z\to c_j$, one gets:
\begin{equation}\label{6}
\frac{\partial f^m}{\partial v_k}(c_j)=
(f^{m-1})'(v_j)\left\{\frac{\partial f}{\partial v_k}(c_j)+
\sum_{n=1}^{m-1}\frac{\frac{\partial f}{\partial v_k}(f^{n-1}(v_j))}
{(f^{n})'(v_j)}\right\}.
\end{equation}
As we know, $\frac{\partial f}{\partial v_k}(c_j)=\delta_{j,k}$.
Besides, $\frac{\partial f}{\partial v_k}(z)=f'(z) L_k(z)=p_k(z)$
is a polynomial
of degree at most $d-2$. Hence, for some constant $C_k$ and all $z$,
\begin{equation}\label{bdd}
|\frac{\partial f}{\partial v_k}(z)|\le C_k (1+|z|^{d-2}).
\end{equation}
Now, assume that $c_j$ is weakly expanding.
As $c_j\in J$, the sequence $\{f^n(v_j)\}_{n\ge 0}$ is
uniformly bounded. Then~(\ref{bdd}) 
and the summability condition~(\ref{summ}) imply that the series
$\sum_{n=1}^{\infty}
\frac{\frac{\partial f}{\partial v_k}(f^{n-1}(v_j))}{(f^{n})'(v_j)}$
converges absolutely. 
Thus we have:
\begin{equation}\label{serl}
L(c_j, v_k)=\lim_{m\to \infty}\frac{\frac{\partial f^m}{\partial v_k}(c_j)}
{(f^{m-1})'(v_j)}=\delta_{j,k}+\sum_{n=1}^{\infty}
\frac{\frac{\partial f}{\partial v_k}(f^{n-1}(v_j))}{(f^{n})'(v_j)}.
\end{equation}
This ends the proof of Part (a) of Theorem~\ref{one}.

The following corollary is immediate from~(\ref{serl}),
\begin{coro}\label{lemravnie}
\begin{equation}\label{lcoro}
L(c_j, v_k)=\delta_{j,k}+A(v_j, v_k),
\end{equation}
for some function $A(x,y)$. 
\end{coro}
\subsection{The Ruelle operator and an operator identity.}\label{sectruelle}
\paragraph{The operator.}

As in the proof of Corollary 1(b) of~\cite{Le},
the main tool for us is 
the following linear operator $T_f$ accosiated
to a rational function $f$, which acts on functions 
as follows:
$$T_f\psi(x)=\sum_{w:f(w)=x}\frac{\psi(w)}{(f'(w))^2},$$
provided $x$ is not a critical value of $f$.

Next statement is about an arbitrary rational
function which fixes infinity.
\begin{lem}\label{formula}
Let $f$ be any rational function so that $f(\infty)=\infty$. 
Let $c_j$, $j=1,...,p$ be all
geometrically different critical points of $f$ lying in the 
complex plane ${\bf C}$, such
that
the corresponding critical values $v_j=f(c_j)$, $j=1,...,p$,
are also in ${\bf C}$. Denote by $m_j$ the multiplicity
of $c_j$, $j=1,...,p$.
Then there are functions $L_1(z),...,L_p(z)$
as follows. For every $z\in {\bf C}$, which is not a critical point of $f$,
for every $x\in {\bf C}$ which is not a critical value of $f$ and such
that $x\not=f(z)$, we have:
\begin{equation}\label{rcauchy1}
T_f\frac{1}{z-x}:=
\sum_{y: f(y)=x}\frac{1}{f'(y)^2}\frac{1}{z-y}=
\frac{1}{f'(z)}\frac{1}{f(z)-x}+
\sum_{j=1}^{p}\frac{L_j(z)}{x-v_j}.
\end{equation}
Furthermore, each function $L_j$ obeys the following two properties:

(1) $L_j$ is a meromorphic function in the complex plane of the form:
\begin{equation}\label{tochno}
L_j(z)=\sum_{i=1}^{m_j}\frac{q^{(j)}_{m_j-i}}{(z-c_j)^i} \ , \ \ \ \ \ \ \
q^{(j)}_{m_j-i}\in {\bf C},
\end{equation}
(2) for every $k=1,...,p$,
the function $f'(z)L_j(z)-\delta_{j,k}$ has zero at the point $c_k$  
of order at least $m_j$.
\end{lem}
\begin{proof}
Fixing $z,x$ as in the lemma, take $R$ big enough and
consider the integral
\begin{eqnarray*}
I=\frac{1}{2\pi i}\int_{|w|=R}\frac{dw}{f'(w)(f(w)-x)(w-z)}.
\end{eqnarray*}
As the integrant is $O(1/w^2)$ at infinity, $I=0$.
On the other hand, applying
the Residue Theorem,
$$I=-T_f\frac{1}{z-x}+\frac{1}{f'(z)}\frac{1}{f(z)-x}+
\sum_{j=1}^{p}I_j(z,x).$$
Here
$$I_j(z,x)=\frac{1}{2\pi i}\int_{|w-c_j|=\epsilon}
\frac{dw}{f'(w)(f(w)-x)(w-z)}.$$
Near $c=c_j$, $f'(w)=(w-c)^{m}r(w)$, where
$m=m_j$ and $r=r_j$ is holomorphic with $r(c)\not=0$.
Denote
$$\frac{1}{r(w)}=\sum_{k=0}^\infty q_k(w-c)^k,$$
where $q_k=q_k^{(j)}$ and 
$q_0=q^{(j)}_0=1/r(c)\not=0$.
We can write:
\begin{eqnarray*}
\frac{1}{f'(w)(f(w)-x)(w-z)} &=& \\ 
\frac{1}{(w-c)^m
r(w)((f(c)-x)+O((w-c)^{m+1}))(w-z)} &=& \\
\frac{1}{r(w)(f(c)-x)(w-z)}\frac{1}{(w-c)^m}+O(w-c) &=& \\
\frac{1}{x-f(c)}\sum_{k=0}^\infty q_k(w-c)^{k-m}\sum_{n=0}^\infty
\frac{(w-c)^n}{(z-c)^{n+1}}+O(w-c).
\end{eqnarray*}
We see from here that
$I_j(z,x)=L_j(z)/(x-f(c_j))$, where
$L_j(z)$ has precisely the form (\ref{tochno}).
Now, consider $\tilde L(z)=f'(z)L_j(z)$. By~(\ref{tochno}),
$\tilde L(z)$
has zero at every $c_k\not=c_j$ of order at least $m_k$.
On the other hand, as $z\to c_j$, then
$$ f'(z)L_j(z)=r_j(z)\left [ q^{(j)}_0+q^{(j)}_1(z-c_j)+...
+q^{(j)}_{m_j-1}(z-c_j)^{m_j-1} \right ]=$$
$$r_j(z)\left [\frac{1}{r_j(z)}-
\sum_{k=m_j}^\infty q^{(j)}_k(z-c_j)^k\right ]=
1-(z-c_j)^{m_j} g(z),$$
where $g$ is holomorphic near $c_j$.
This finishes the proof of the property (2).
\end{proof}
As a simple corollary of the last two statements we have:
\begin{prop}\label{param}
Let $f\in \Pi_{d,\bar p}$. Then
\begin{equation}\label{parameq}
T_f\frac{1}{z-x}=
\frac{1}{f'(z)}\frac{1}{f(z)-x}+
\sum_{k=1}^{p}\frac{\frac{\partial f}{\partial v_k}(z)}{f'(z)}
\frac{1}{x-v_k}.
\end{equation}
\end{prop}
\begin{proof}
By the property (1) of Lemma~\ref{formula}, $f'(z) L_k(z)$  
is a polynomial of degree at most $d-2$, which, by the property (2)
of Lemma~\ref{formula},
coincides with the polynomial $p_k(z)=\frac{\partial f}
{\partial v_k}(z)$ introduced
in Proposition~\ref{derlem}. Therefore, indeed, 
$L_k(z)=\frac{\frac{\partial f}{\partial v_k}(z)}{f'(z)}$.
\end{proof}
\paragraph{The operator identity and its corollary.}
\begin{prop}\label{main}
Let $f\in \Pi_{d, \bar p}$. We have (in formal series):
\begin{equation}\label{old}
\varphi_{z,\lambda}(x)-\lambda (T_f\varphi_{z,\lambda})(x)=\frac{1}{z-x}+
\lambda\sum_{k=1}^{p}\frac{1}{v_k-x}
\Phi_{k}(\lambda, z),
\end{equation}
where 
\begin{equation}\label{varphi}
\varphi_{z,\lambda}(x)=\sum_{n=0}^\infty\frac{\lambda^n}{(f^n)'(z)}
\frac{1}{f^n(z)-x},
\end{equation}
and
\begin{equation}\label{Psi}
\Phi_{k}(\lambda, z)=\sum_{n=0}^\infty \lambda^{n+1}
\frac{\frac{\partial f}{\partial v_k}(f^n(z))}{(f^{n+1})'(z)},
\end{equation}
for $\lambda$ complex parameter and $z,x$ complex variables.
\end{prop}
\begin{proof}
We use Proposition~\ref{param} and write (in formal series):
$$\varphi_{z,\lambda}(x)-\lambda (T_f\varphi_{z,\lambda})(x)=$$
$$\sum_{n=0}^\infty\frac{\lambda^n}{(f^n)'(z)}\frac{1}{f^n(z)-x}-
\lambda\sum_{n=0}^\infty\frac{\lambda^n}{(f^n)'(z)}
\left\{\frac{1}{f'(f^n(z))}\frac{1}{f^{n+1}(z)-x}+
\sum_{k=1}^p \frac{\frac{\partial f}{\partial v_k}(f^n(v_j))}{f'(f^n(v_j))}
\frac{1}{x-v_k}\right\}$$
$$=\frac{1}{z-x}+
\sum_{k=1}^p \frac{1}{v_k-x}\sum_{n=0}^\infty \lambda^{n+1}
\frac{\frac{\partial f}{\partial v_k}(f^n(v_j))}{(f^{n+1})'(v_j)}.$$
\end{proof}
\begin{com}\label{iden}
If $c_k$ is a simple critical point,
$\Phi_{k}(\lambda, z)=\frac{\lambda}{f''(c_k)}\varphi_{z,\lambda}(c_k)$,
and with $\Phi_{k}(\lambda, z)$ in such form, equation
(\ref{old}) above as well as (\ref{oldrat}) from 
Proposition~\ref{mainrat} of the next Section
appeared for the first time
in~\cite{LSY} (where it is written
via the resolvent $(1-\lambda T_f)^{-1}$).
The only new (and crucial) ingredient of Proposition~\ref{main} (as well as
Proposition~\ref{mainrat}) is the
representation of the coefficients $\Phi_{k}(\lambda, z)$
in~(\ref{old}) via derivatives with respect to the local coordinates
in the appropriate space of maps.
\end{com}
Putting in Proposition~\ref{main} $\lambda=1$ and $z=v_j$ and
combining it with~(\ref{serl}), see the proof of Theorem~\ref{one} (a)
in Section~\ref{pra}, we get:
\begin{prop}\label{vaznoe}
Let $c_j$ be a summable critical point of $f\in \Pi_{d, \bar p}$.
Then
\begin{equation}\label{vazeq}
H_j(x)-(T_f H_j)(x)=\sum_{k=1}^{p}\frac{L(c_j, v_k)}{v_k-x},
\end{equation}
where 
\begin{equation}\label{hj}
H_j(x)=\sum_{n=0}^\infty\frac{1}{(f^n)'(v_j)(f^n(v_j)-x)},
\end{equation}
\end{prop}
\subsection{Proof of Part (b) of Theorem~\ref{one}}\label{oneb}
The proof is very similar to the one of Corollary 1(b) of
\cite{Le},
where the family of unicritical polynomials is considered
(see also~\cite{mak},~\cite{mult} and references therein).
But some additional considerations are needed, if the 
number of summable critical points is more than one
and their critical values coincide.
Denote by $S$ the set of indexes of given collection
of $r$ summable critical points of $f$. 
Assume the contrary, i.e., the rank of the matrix ${\bf L}$ is less than 
$r$.
This holds if and only if there exist numbers $a_j$, for $j\in S$,
which are not all zeros, such that, for every $1\le k\le p$,
\begin{equation}\label{rankl}
\sum_{j\in S} a_j L(c_j, v_k)=0.
\end{equation}
Then, by Proposition~\ref{vaznoe}, 
\begin{equation}\label{hhj}
H=\sum_{j\in S} a_j H_j 
\end{equation} 
is an integrable fixed point of $T_f$ which is 
holomorphic in each component of the complement
${\bf C}\setminus J$. Let us show that $H=0$ off $J$.
We use that
$T_f$ is weakly contracting. Consider a component $\Omega$
of ${\bf C}\setminus J$.
If $\Omega$ is not a Siegel disk, then considering the backward orbit 
$\cup_{n\ge 0}f^{-n}(\Omega)$ it is easy to find its open subset $U$,
such that $f^{-1}(U)\subset U$
and $U\setminus f^{-1}(U)$
is a non-empty open subset of $\Omega$. 
Since $H=T_f H$, we then have
(the integration is against the Lebesgue measure on the plane):
$$
\int_U |H(x)|d\sigma_x=\int_U \left|\sum_{f(w)=x}
\frac{H(w)}{f'(w)^2}\right|d\sigma_x\le $$
$$\int_U \sum_{f(w)=x}
\frac{|H(w)|}{|f'(w)|^2}d\sigma_x=
\int_{f^{-1}(U)} |H(x)|d\sigma_x,
$$
which is possible only if $H=0$ in $U\setminus f^{-1}(U)$,
hence, in $\Omega$, because $U\subset \Omega$. 
And if $\Omega$ is a Siegel disk, 
we proceed as
in the proof of Corollary 1(b) p. 190 of~\cite{Le} to show that
$H=0$ in $\Omega$ as well (see also Lemma~\ref{contr}). 
Thus, $H=0$ off $J$.
On the other hand, $H$ can be represented as 
\begin{equation}\label{hform}
H(x)=\sum_{k=0}^\infty \frac{\alpha_k}{x-b_k}, 
\end{equation}
where the points $b_k\in J$ are pairwise different,
and $|\alpha_k|<\infty$. 
Consider a measure with compact support
$\mu=\sum_{k=0}^\infty \alpha_k \delta(b_k)$, where
$\delta(z)$ is the Dirac measure at the point $z$. 
Then
$H=0$ off $J$ implies that
the measure $\mu$ annihilates
any function which is holomorphic in a neighborhood of $J$.
Indeed, for any such function $r(z)$ and an appropriate 
contour $\gamma$ enclosing $J$,
$$
\int r(z)d\mu(z)=\int \left(\frac{1}{2\pi i}\int_\gamma
\frac{r(w)}{w-z}dw\right)d\mu(z)=
\frac{1}{2\pi i}\int_\gamma r(w)\left(\int\frac{d\mu(z)}{w-z}\right)dw=
$$
$$
\frac{1}{2\pi i}\int_\gamma H(w)r(w)dw=0.
$$
As every point of $J$ belongs also to the boundary of the basin of infinity, 
by a corollary from Vitushkin's theorem
(see e.g.~\cite{Ga}), every continuous function on $J$
is uniformly approximated by rational functions.
It follows, $\mu=0$, in other words, the representation~(\ref{hform})
is trivial:
\begin{equation}\label{mu0}
\alpha_k=0, \ \ k=0,1,2,... .
\end{equation}
That is to say, the left hand side of~(\ref{hhj}) is zero:
\begin{equation}\label{h=0}
\sum_{j\in S} a_j H_j=0.
\end{equation}
But, if the number of indexes in $S$ is bigger than one, 
it does not imply immediately that all the numbers
$a_j$ in~(\ref{h=0}) must vanish. Indeed, some $H_j$ can even coincide:
by the definition of the function $H_j$, see Proposition~\ref{vaznoe},
$H_{j_1}=H_{j_2}$, if $v_{j_1}=v_{j_2}$.
So, we will use~(\ref{rankl}) along with~(\ref{h=0}) to prove  
\begin{lem}\label{stranno}
$a_k=0$, for $k\in S$.
\end{lem}
\begin{proof}
As the first step, we show that {\it different} functions $H_j$
are linearly independent.
Let us denote
by $V_i$, $1\le i\le q$, all the {\it different} critical
values of $f$. For $1\le i\le q$, introduce $\tilde H_i$
to be $H_j$, for every $j$, such that $v_j=V_i$. 
The set of indexes $S$
is a disjoint union of subsets $S_i$, $1\le i\le e$,
so that
$j\in S_i$ if and only if $v_j=V_i$. 
Then
\begin{equation}\label{tildeh}
0=\sum_{i=1}^e \tilde a_i \tilde H_i, \ \ \ \
\tilde a_i=\sum_{j\in S_i} a_j.
\end{equation}
We show that this representation is trivial:
$\tilde a_i=0, \ \ \ 1\le i\le e$.
By Proposition~\ref{vaznoe},
\begin{equation}\label{tildeh1}
\tilde H_i(x)=\sum_{n=0}^\infty\frac{1}{(f^n)'(V_i)(f^n(V_i)-x)},
\end{equation}
where $\{V_i\}$ are pairwise different critical values of $f$.
By contradiction, assume that some $\tilde a_i\not=0$.
Without loss of generality, one can assume further that 
$i=1$, i.e., $\tilde a_1\not=0$.
{\it Claim:} $V_1$ is a point of a periodic orbit $P$ of $f$,
and, moreover, if $P$ contains some $f^k(V_j)$, $k\ge 0$, 
with $\tilde a_j\not=0$, then $V_j\in P$. 
By~(\ref{tildeh}), the residue $1$ of the first term
in the sum~(\ref{tildeh1}) (with $i=1$)
must be cancelled. Hence, there is a critical value $V_{i_1}$, so that
$V_1=f^{k_1}(V_{i_1})$ with $\tilde a_{i_1}\not=0$, 
for some minimal $k_1>0$. If here $i_1\not=1$,
then, by the same reason, there is $V_{i_2}$, so that
$V_{i_1}=f^{k_2}(V_{i_2})$ with $\tilde a_{i_2}\not=0$, 
for some minimal $k_2>0$. We continue until we arrive at
a critical value that has been met before. This proves the {\it Claim}.

Denote by $V_{i_0}=V_1,V_{i_1},...,V_{i_s}$ all the different
critical values which are points of the periodic orbit $P$
and such that $\tilde a_{i_l}\not=0$, $0\le l\le s$.
Since $V_{i_l}$ are periodic points, each
$\tilde H_{i_l}$, $0\le l\le s$, is a rational function
(see also~(\ref{dir2}) below). We have:
\begin{equation}\label{dir0}
\sum_{l=0}^s \tilde a_{i_l} \tilde H_{i_l}=0,
\end{equation}
where every $\tilde a_{i_l}$ is non zero.
Denote by $T$ the (minimal) period of $P$.
Changing indexes, if necessary, there are 
some $0<n_1<n_{2}<...<n_s<T$, such that 
\begin{equation}\label{ind}
f^{n_l}(V_1)=V_{i_l},\ \ \ \ 1\le l\le s. 
\end{equation}
The rest is a direct not difficult
calculation. From~(\ref{tildeh1}) and~(\ref{ind}), for $1\le l\le s$,
\begin{equation}\label{dir1}
\tilde H_{i_l}(x)=(f^{n_l})'(V_1)\left(\tilde H_{1}(x)-
\sum_{n=0}^{n_l-1}\frac{1}{(f^n)'(V_1)(f^n(V_1)-x)}\right),
\end{equation}
while, if $\rho=(f^T)'(V_1)$,
\begin{equation}\label{dir2}
\tilde H_{1}(x)=\frac{1}{1-\rho^{-1}}\sum_{n=0}^{T-1}
\frac{1}{(f^n)'(V_1)(f^n(V_1)-x)}.
\end{equation}
(Note that $|\rho|>1$, because $V_1$ is summable.)
If $s=0$ (i.e., $V_1$ is the only critical value
$V_i\in P$ such that $\tilde a_i\not=0$), 
then $\tilde H_{1}=0$, which is a contradiction with~(\ref{dir2}).
Assume $s>0$.
By~(\ref{dir1}), the relation~(\ref{dir0}) turns into:
\begin{equation}\label{dir3}
\left(\tilde a_1+\sum_{l=1}^s \tilde a_{i_l}(f^{n_l})'(V_1)\right)
\tilde H_{1}(x)=\sum_{l=1}^s
\tilde a_{i_l}(f^{n_l})'(V_1)\sum_{n=0}^{n_l-1}
\frac{1}{(f^n)'(V_1)(f^n(V_1)-x)},
\end{equation}
where $\tilde H_1$ is given by~(\ref{dir2}).
Since $0\le n_1-1<...<n_s-1\le T-2$ and the points $\{f^n(V_1)\}_{n=0}^{T-1}$
are pairwise different, 
the function on the right hand side of~(\ref{dir3})
has no pole at $f^{T-1}(V_1)$ while $\tilde H_1$ has.
Hence, $\tilde a_1+\sum_{l=1}^s \tilde a_{i_l}(f^{n_l})'(V_1)=0$.
It means that the left hand side in~(\ref{dir3}) is zero
identically. Hence, the function in the right hand side 
of~(\ref{dir3}) has no pole
at $f^{n_s-1}(V_1)$, which is possible only if $\tilde a_{i_s}=0$.
This is already a contradiction, because, by our assumption,
all $\tilde a_{i_l}$ are non-zero numbers.

The contradiction shows that
$\tilde a_i=0$, for $1\le i\le e$. 
To get from this that all $a_k=0$, we use Corollary~\ref{lemravnie},
and rewrite~(\ref{rankl}), 
for every $k\in S$:
\begin{eqnarray*}\label{fin}
0=\sum_{j\in S} a_j L(c_j, v_k)=
\sum_{i=1}^e \sum_{j\in S_i} a_j \left(\delta_{j,k}+A(V_i, v_k)\right) &=& \\
\sum_{i=1}^e 
\left(\sum_{j\in S_i} a_j\delta_{j,k} + A(V_i, v_k)\sum_{j\in S_i} a_j\right)
=\sum_{i=1}^e 
\left(\sum_{j\in S_i} a_j\delta_{j,k} + A(V_i, v_k)\tilde a_i\right)
&=& \\
\sum_{i=1}^e\sum_{j\in S_i} a_j\delta_{j,k}=\sum_{j\in S} a_j\delta_{j,k}=a_k.
\end{eqnarray*}
\end{proof}
This ends the proof of Theorem~\ref{one} (b), because this 
is a contradiction with the fact, that at least one $a_j$, $j\in S$,
is a non zero number.
\section{Rational functions}\label{rat}
\subsection{Local spaces of rational maps}
Let $f$ be a rational function of degree $d\ge 2$, such that 
\begin{equation}\label{rrr}
f(z)=\sigma z+b+\frac{P(z)}{Q(z)},
\end{equation}
where $\sigma\not=0,\infty$, and $Q$, $P$ are polynomials of
degrees $d-1$ and at most $d-2$ respectively, which have no common roots.
Without loss of generality, one can assume that $Q(z)=z^{d-1}+A_1
z^{d-2}+...+A_{d-1}$ and $P(z)=B_0 z^{d-2}+...+B_{d-2}$.

Let $p'$ stand for the number of distinct critical points
$c_1,...,c_{p'}$ of $f$. Note that all of them are different from $\infty$.
Denote by $m_j$ the multiplicity of $c_j$,
that is, the equation $f(w)=z$ has precisely $m_j+1$ different
solutions near $c_j$ for $z$ near $f(c_j)$ and $z\not=f(c_j)$, $j=1,...,p'$.
Observe that it is equivalent to say that, 
if $f=\frac{\hat P}{Q}$
where $\hat P(z)=Q(z)(\sigma z+b)+P(z)$, then 
$c_j$, for $1\le j\le p'$,
is a root of the polynomial $\hat P' Q - \hat P Q'$ of multiplicity
$m_j$.
Note that $\sum_{j=1}^{p'} m_j= 2d-2$. 

Let $\bar p'=\{m_j\}_{j=1}^{p'}$ denote
the vector of multiplicities.
Let $v_j=f(c_j)$, $1\le j\le p'$, be the corresponding critical values.
We assume that some of them can coincide as well as 
some can be $\infty$.
By $p$ we denote the number of critical points of $f$,
so that their images (i.e., the corresponding critical values)
lie in the complex plane.

As in \cite{mult}, we define a local (near $f$) space 
$\Lambda_{d, \bar p'}$ of rational functions.
\begin{defi}\label{ratpoldef}
A rational
function $g$ of degree $d$ belongs to $\Lambda_{d, \bar p'}$ if
and only if it has the form
\begin{equation}\label{rrg}
g(z)=\sigma(g) z+b(g)+\frac{P_g(z)}{Q_g(z)},
\end{equation}
where the numbers $\sigma(g), b(g)$, and 
the polynomials $Q_g=z^{d-1}+A_1(g)z^{d-2}+...+A_{d-1}(g)$ 
and $P_g(z)=B_0(g) z^{d-2}+...+B_{d-2}(g)$
are close to $\sigma, b$, $P, Q$ respectively.
Moreover, $g$ has $p'$ different critical points $c_1(g),...,c_{p'}(g)$,
so that $c_j(g)$ is close to $c_j$ and has the same multiplicity $m_j$,
$1\le j\le p'$.
\end{defi}

For $g$ in a sufficiently small
neighborhood of $f$ in $\Lambda_{d, \bar p'}$, 
introduce a vector $\bar v(g)\in {\bf C}^{p'+2}$ as follows.
Let us fix an order $c_1,...,c_{p'}$ in the collection 
of all critical points of $f$. We will do it in such a way,
that the first $p$ indexes correspond to finite critical values,
i. e. $v_j\not=\infty$ for $1\le j\le p$ and $v_j=\infty$
for $p<j\le p'$ (if $p<p'$).
Remark that $p\ge 1$: there always at least one
critical value in the plane.
There exist $p'$ 
functions $c_1(g),...,c_{p'}(g)$ which are defined and continuous
in a small neighborhood of $f$ in $\Lambda_{d, \bar p'}$,
such that they constitute all different critical points of $g$.
Moreover, $c_j(g)$ has the multiplicity $m_j$, $1\le j\le p'$. 
Define now the vector $\bar v(g)$.
If all critical values of $f$ lie in the plane, then 
we set
$\bar v(g)=\{\sigma(g), b(g), v_1(g),...,v_{p'}(g)\}$.
If some of the critical values $v_j$ of $f$ are infinity,
that is, $v_j=\infty$ for $p<j\le p'$, then we replace
in the definition of $\bar v(g)$ corresponding $v_j(g)$
by $1/v_j(g)$:
$$\bar v(g)=\left\{\sigma(g), b(g), v_1(g),...,v_{p}(g),
\frac{1}{v_{p+1}(g)},...,\frac{1}{v_{p'}(g)}\right\}.$$
In particular,
$\bar v=\bar v(f)=\{\sigma, b, v_1,...,v_p,0,...,0\}$.

We can identify $g\in \Lambda_{d, \bar p'}$ as above with the point
$$\bar g=\left\{\sigma(g), b(g), A_1(g),...,A_{d-1}(g), 
B_0(g),...,B_{d-2}(g)\right\}$$
of ${\bf C}^{2d}$. Then $\Lambda_{d, \bar p'}$ 
is an analytic
variaty in ${\bf C}^{2d}$. We denote it again by $\Lambda_{d, \bar p'}$.
Every critical point $c_j$ of $f$ of multiplicity $m_j$ is determined
by $m_j-1$ algebraic equations. Therefore, $\Lambda_{d, \bar p'}$
is defined by $\sum_{j=1}^{p'} (m_j-1)=2d-2-p'$ equations corresponding
to the critical points.
Thus the (complex) dimension of $\Lambda_{d, \bar p'}$ is at least 
$2d-\sum_{j=1}^{p'} (m_j-1)=p'+2$.
In fact,
as it is proved in \cite{mult}, Sect. 7:
\begin{prop}\label{rlocal}
$\Lambda_{d, \bar p'}$
is a complex-analytic manifold of dimension $p'+2$,
and $\bar v(g)$ defines a local holomorphic
coordinate of $g\in \Lambda_{d, \bar p'}$.
\end{prop}

\

Two rational functions are called close if they are uniformly
close in the Riemann metric on the sphere.
We call two rational functions ($M$-){\it equivalent} if there is
a Mobius transformation $M$ which conjugates them.
Every rational function $f$ is equivalent to some $\tilde f$ 
of degree $d\ge 2$ be of the form~(\ref{rrr}). Indeed, 
$f$ has either a repelling fixed point,
or a fixed point with the multiplier $1$ (see
e.g.~\cite{mil}). Hence, there exists a Mobius
transformation $P$, such that $\infty$ is a fixed 
non-attracting point of $\tilde f=P\circ f\circ P^{-1}$. See also 
Section~\ref{disc}.

If a critical value of $f$ or its iterate is infinity, 
we consider also another space of rational maps which is
biholomorphic to $\Lambda_{d, \bar p'}$.
Let us fix a Mobius transformation $M$, such that $\alpha=M(\infty)$ 
lies outside of the critical orbits $\{f^n(c_k): n\ge 0,
1\le k\le p'\}$ of $f$. Then we make the same change of variable $M$ for all
maps from $\Lambda_{d, \bar p'}$, and obtain a new space with
a natural coordinate: 
\begin{defi}\label{Mspace}
Let $\Lambda^M_{d, \bar p'}$ be the set of maps
$$\{M^{-1}\circ g\circ M: g\in \Lambda_{d, \bar p'}\}.$$
If $\tilde g\in \Lambda^M_{d, \bar p'}$, the
points $c_k(\tilde g)=M^{-1}(c_k(g))$ and $v_k(\tilde g)=M^{-1}(v_k(g))$
are the critical point 
(of multiplicity $m_k$) and
the critical value of $\tilde g$, $1\le k\le p'$, respectively. 
Then the vector 
\begin{equation}\label{vm}
\bar v^M(\tilde g)=\{\sigma(g), b(g), v_1(\tilde g), v_{2}(\tilde g),...,
v_{p'}(\tilde g)\}
\end{equation}
defines a holomorphic coordinate system
in $\Lambda^M_{d, \bar p'}$.
\end{defi}
The advantage of the new space is that the critical orbits
of $\tilde f=M^{-1}\circ f\circ M$ lie in the plane
(although can be unbounded). 
\subsection{Subspaces}\label{disc}
Suppose $f$ is an arbitrary rational function of degree $d\ge 2$.
Denote by $p'$ 
the number of different critical points of $f$
in the Riemann sphere, and by $\bar p'$ the vector of multiplicities
at the critical points.
As it was mentioned, there is an alternative:
either 

{\bf (H)}: {\it $f$ has a fixed point $a$, such that $f'(a)\not=0,1$,}

or 

{\bf (N)}: {\it the multiplier of
every fixed point of $f$ is either $0$ or $1$,
and there is a fixed point with the multiplier $1$.} 

\

The case {\bf (N)}
is degenerate. We consider each case separately.

{\bf (H)}. Let $P$ be a Mobius transformation, such that $P(a)=\infty$.
Then $\tilde f=P\circ f\circ P^{-1}$
belongs to $\Lambda_{d, \bar p'}$. Moreover,
$P$ can be chosen uniquely in such a way, that $b(\tilde f)=0$,
and the critical value $v_p$ of
$\tilde f$ is equal to $1$. 
Let us define a submanifold $\Lambda_{\tilde f}$ 
of $\Lambda_{d, \bar p'}$
consisting of $g\in \Lambda_{d, \bar p'}$
in a neighborhood of $\tilde f$, such that $b(g)=0$, 
and $v_p(g)=1$.
The coordinate $\bar v(g)$ in $\Lambda_{d, \bar p'}$ restricted
to $\Lambda_{\tilde f}$ is obviously a coordinate in that subspace,
which turns it into a $p'$-dimensional complex manifold.

\

{\bf (N)}. There are two sub-cases to distinguish.

{\bf (NN)}: $f$ has a fixed point $a$, such that 
$f'(a)=1$ and $f''(a)\not=0$. 
Let $P$ be a Mobius transformation, such that $P(a)=\infty$.
Then $\tilde f=P\circ f\circ P^{-1}$
belongs to $\Lambda_{d, \bar p'}$. Moreover,
$P$ can be chosen uniquely in such a way, that 
$v_p(\tilde f)=1$ and $b(\tilde f)=1$. 
Then we define $\Lambda_{\tilde f}$ to be the set of all
$g\in \Lambda_{d, \bar p'}$
in a neighborhood of $\tilde f$, such that $b(g)=1$, 
and $v_{p}(g)=1$.
Coordinates in $\Lambda_{\tilde f}$ is defined as in the previous case,
and it turns $\Lambda_{\tilde f}$ into a $p'$-dimensional complex manifold.

\

{\bf (ND)}: every fixed point with multiplier $1$
is degenerate. Let $a$ be one of them: $f'(a)=1$ and
$f''(a)=0$. Then the Mobius map $P$ can be chosen
uniquely in such a way, that  
$\tilde f(z)=P\circ f\circ P^{-1}(z)=
z+O(1/z)$, and $\tilde f$ has a critical value equal to $1$
in one attracting petal of $\infty$, and equal to $0$
in another attracting petal of $\infty$. 
Then $\Lambda_{\tilde f}$ consists of $g\in \Lambda_{d, \bar p'}$
in a neighborhood of $\tilde f$, such that the 
critical value of $g$ which is close to $v_{p-1}(\tilde f)=1$ 
is identically equal to
$1$, and the critical value of $g$, which is close to $v_{p}(\tilde f)=0$,
is identically equal to $0$. 
Then $\Lambda_{\tilde f}$ is a $p'$-dimensional complex manifold.

\

It is easy to check that in any of these cases,
every $\tilde g\in \Lambda_{d, \bar p'}$
is equivalent (by a linear conjugacy) to some 
$\tilde g_1\in \Lambda_{\tilde f}$. Let us drop the condition
that maps fix infinity and consider the set $X_{\tilde f}$ of {\it all}
rational functions $\hat g$ of degree $d$ which are close to $\tilde f$
and such that $\hat g$ has $p'$ different critical points
with the same corresponding multiplicities. Then, since any such $\hat g$
has a fixed point close to infinity, it is equivalent
to some $\tilde g\in \Lambda_{d, \bar p'}$ and, hence, to some
$\tilde g_1\in \Lambda_{\tilde f}$.
In any of the cases {\bf (H)}, {\bf (NN)}, {\bf (ND)}, 
we denote 
\begin{equation}\label{XL}
X_f=\{g=P^{-1}\circ \hat g\circ P: 
\hat g\in X_{\tilde f}\}, \ \ \
\Lambda_f=\{g=P^{-1}\circ \tilde g\circ P: 
\tilde g\in \Lambda_{\tilde f}\},
\end{equation} 
where the Mobius map $P$ is taken as before.
This shows the first part of the following statement. 
\begin{prop}\label{gen}
(a) Every rational function $f$ of degree $d\ge 2$ is equivalent
to some $\tilde f\in \Lambda_{d, \bar p'}$, where
$\tilde f$ is of one and only one type:
either ${\bf (H)}$ or ${\bf (NN)}$ or ${\bf (ND)}$.
If $X_f$ denotes the set of all
rational functions $g$ of degree $d$ which are close to $f$
and such that $g$ have $p'$ distinct critical points
with the same corresponding multiplicities, then
any $g\in X_f$ 
is equivalent to some $g_1\in \Lambda_{f}$. 

(b) In the case {\bf (H)}, the set $X_f$ is a complex-analytic 
manifold of dimension $p'+3$, and the correspondence
$\pi: g\in X_f\mapsto g_1\in \Lambda_f$ is a well-defined holomorphic map.
\end{prop}
Let us show (b). In this case
$\sigma(\tilde f)=\tilde f'(\infty)\not=1$, hence,
any $\hat g\in X_{\tilde f}$ has a unique fixed point $\beta$ which
is close to $\infty$. This defines a one-to-one map ${\bf B}$
between $X_{\tilde f}$ and $\Lambda_{d, \bar p'}$ by
${\bf B}: \hat g\mapsto \tilde g=M_\beta\circ \hat g\circ M_\beta^{-1}$, where
$M(z)=\beta z/(\beta-z)$. It is also easy to see that in this case
there is a natural map
${\bf A}: \tilde g\in \Lambda_{d, \bar p'}\mapsto
\tilde g_1\in \Lambda_{\tilde f}$ defined by
$\tilde g_1=A^{-1}\circ \tilde g\circ A$, where 
$A(z)=kz+e$ with $k=v_p(\tilde g)+\frac{b(\tilde g)}{\sigma(\tilde g)-1}$
and $e=-\frac{b(\tilde g)}{\sigma(\tilde g)-1}$.
The vector $\bar v(\hat g)=\{\frac{1}{\beta}, \bar v(\tilde g)\}\in 
{\bf C}^{p'+3}$ defines a complex-analytic structure on $X_{\tilde f}$,
and the map ${\bf A}\circ {\bf B}: X_{\tilde f}\to \Lambda_{\tilde f}$ 
is holomorphic.
Since, by~(\ref{XL}) above, $X_f$, $\Lambda_f$ are isomorphic to 
$X_{\tilde f}$, $\Lambda_{\tilde f}$, the part (b) follows.
Note in conclusion that the Mobius conjugacy between $g\in X_f$ and
$g_1\in \Lambda_f$, which we just constructed, tends to the
identity map uniformly on the Riemann sphere as $g$ tends to $f$.

\subsection{Main result}\label{mainres}
\begin{defi}\label{ccomdef}
A compact subset $K$
of the Riemann sphere is called {\it C-compact}, if there is a Mobius
transformation $M$, such that $M(K)$
is a compact subset of the complex plane with the property that
every continuous function on $M(K)$ can be
uniformly approximated by functions which are
holomorphic in a neighborhood of $M(K)$.
\end{defi}
Clearly, C-compact sets must have empty interior.
Vitushkin's theorem characterizes such compacts on the plane,
see e.g. \cite{Ga}.
As a simple corollary of this theorem, we have that
each of the following conditions is sufficient
for $K$ to be a C-compact set:

(1) $K$ has Lebesgue measure zero,

(2) every $z\in K$ belongs to the boundary
of a component of the complement of $K$.

We call a rational function $f$ {\it exceptional} if
$f$ is {\it double covered by an integral torus
endomorphism} (also known as {\it a flexible Lattes map}): 
these form a family of explicitly described 
critically finite rational maps with Julia sets which are
the whole Riemann sphere,
see, for example,~\cite{dht},~\cite{mcm},~\cite{mill}.

\

Let $f\in \Lambda_{d, \bar p'}$.
Fix a Mobius transformation $M$, such that $M(\infty)$ is disjoint
with the forward orbits of the critical points of $f$.
(If $f^n(c_j)\not=\infty$ for all $1\le j\le p'$ and all $n\ge 0$,
one can put $M$ to be the identity map.) 
Consider $\tilde f=M^{-1}\circ f\circ M$ in the space $\Lambda^M_{d, \bar p'}$
with the coordinate $\bar v^M$ (see Definition~\ref{Mspace}). 
We denote by $\tilde c_k=M^{-1}(c_k)$, $\tilde v_k=M^{-1}(v_k)$,
$1\le k\le p'$, the corresponding critical points and
critical values of $\tilde f$. 
As usual, 
$\frac{\partial (\tilde f^l (\tilde c_j))}{\partial \tilde v_k}$
etc. means $\frac{\partial (\tilde g^l (c_j(\tilde g)))}
{\partial v_k(\tilde g)}$ etc. evaluated at $\tilde g=\tilde f$.
\begin{theo}\label{onerat}

(a) Let $c_j$ be a weakly expanding critical point 
of $f$. 
Then,
for every $k=1,...,p'$, 
the following limits exist:
\begin{equation}\label{limratM}
L^M(c_j, v_k):=\lim_{l\to \infty}
\frac{\frac{\partial (\tilde f^l (\tilde c_j))}{\partial \tilde v_k}}
{(\tilde f^{l-1})'(\tilde v_j)},
\end{equation}
\begin{equation}\label{limsigmaM}
L^M(c_j, \sigma):=\lim_{l\to \infty}
\frac{\frac{\partial (\tilde f^l (\tilde c_j))}{\partial \sigma}}
{(\tilde f^{l-1})'(\tilde v_j)}, \ \ \ \ \
L^M(c_j, b):=\lim_{l\to \infty}
\frac{\frac{\partial (\tilde f^l (\tilde c_j))}{\partial b}}
{(\tilde f^{l-1})'(\tilde v_j)}.
\end{equation}
Moreover, if $p<p'$, i.e., $f$ has a critical value at infinity),
then, for $p+1\le j\le p'$ (i.e., when $v_j=\infty$) and
$1\le k\le p'$,
\begin{equation}\label{limratMinf}
L^M(c_j, v_k)=\delta_{j,k}, 
\end{equation}
\begin{equation}\label{limratMinf1}
L^M(c_j, \sigma)=L^M(c_j, b)=0.
\end{equation}

(b) Suppose that $f$ is not exceptional,
and $f$ has a collection of $r$ weakly expanding
critical points. Without loss of generality, one can assume
that $c_{j_1},...c_{j_r}$, $1\le j_1<...<j_r\le p'$,
are these points, and they are ordered as follows.
Denote by $\nu$ the number of such points with the corresponding
critical values to be in the complex plane.
Then denote by
$1\le j_1<...<j_\nu\le p$ the indexes of
these points, and, if $\nu<r$,
i.e., there are critical points from the collection
with the corresponding critical values at infinity, then
$r-\nu$ indexes of these points are the last ones:
$j_{\nu+1}=p'-(r-\nu-1),  j_{\nu+2}=p'-(r-\nu-2),...,  j_{r}=p'$.

Denote by $K$ the union of the $\omega$-limit sets
of these critical points.
Assume that $K$ 
is a C-compact.
Then the rank of the matrix ${\bf L^M}$ (defined shortly) is equal to $r$.

($H_\infty$). Let $f$ be of the type ${\bf (H)}$, i.e.,
$f(z)=\sigma z+O(1/z)$ as $z\to \infty$, where $\sigma\not=1$,
and also $v_p=1$. Then
\begin{equation}\label{rmdeltacutM}
{\bf L^M}=\left(L^M(c_{j_i}, \sigma), L^M(c_{j_i}, v_1),..., 
L^M(c_{j_i}, v_{p-1}), L^M(c_{j_i}, v_{p+1}),...,L^M(c_{j_i}, v_{p'})\right)_{1\le i\le r}
\end{equation}

($NN_\infty$). If $f$ of the type ${\bf (NN)}$, i.e., 
$\sigma=1$, $b\not=0$, and $v_p=1$, then
\begin{equation}\label{rmdeltacutnM}
{\bf L^{M}}=\left(L^M(c_{j_i}, v_1),..., L^M(c_{j_i}, v_{p-1}), 
L^M(c_{j_i}, v_{p+1}),..., L^M(c_{j_i}, v_{p'})\right)_{1\le i\le r}
\end{equation}

($ND_\infty$). Finally, if $f$ of the type ${\bf (ND)}$:
$\sigma=1$, $b=0$, $v_{p-1}=1, v_p=0$, then 
\begin{equation}\label{rmdeltacutnnM}
{\bf L^{M}}=\left(L^M(c_{j_i}, v_1),..., 
L^M(c_{j_i}, v_{p-2}), L^M(c_{j_i}, v_{p+1}),...,
L^M(c_{j_i}, v_{p'})\right)_{1\le i\le r}
\end{equation}
\end{theo}

\

{\bf Proof of Theorem~\ref{tworatintro}}.

We apply the part (b) of Theorem~\ref{onerat}
exactly as we apply Theorem~\ref{one}
in the Comment~\ref{onecoro}, and then apply Proposition~\ref{gen}.
This ends the proof.

\

{\bf Proof of Corollary~\ref{misiu}}

Since 
$K=\cup_{j=1}^{2d-2}\omega(c_j)$ is a hyperbolic set for $f$, and,
hence, is of measure zero, Theorem~\ref{tworatintro} applies.
As every critical point of $f$ is simple and summable,
$p'=r=2d-2$, i.e., $\Lambda=\Lambda_f$, 
and $X_f$ is the (local) space
of {\it all} close to $f$ rational functions of degree $d$. 
Let $\bar x=(x_1,...,x_{2d-2})$ be a local coordinate in $\Lambda$.
Since all critical points of $f$ are summable, $f$ has no 
parabolic periodic orbits. Hence, $f$ is of type ${\bf H}$,
and the
part (b) of Proposition~\ref{gen} applies:
$X_f$ is a complex-analytic manifold, and the 
correspondence $\pi: g\in X_f\mapsto \hat g\in \Lambda_f$ 
(from Proposition~\ref{gen} (a))
is a holomorphic map. Observe, that the restriction
of $\pi$ on $N\subset X_f$ is one-to-one, because otherwise
two different $g_1, g_2\in N$ would be equivalent to the same
$\pi(g_1)=\pi(g_2)$. Hence, the restriction 
$\pi|_N: N\to \Lambda_f$ is a local analytic isomorphism.
This is equivalent to say that, 
in the coordinates $\lambda$ of $N$ and $\bar x$ of $\Lambda$, 
the derivative of $\pi|_N$ at $\lambda_0$
is a non-degenerate $(2d-2)\times (2d-2)$ matrix $\pi'_0$.
In turn, the part (b) of Theorem~\ref{tworatintro}
is equivalent to the statement that the matrix
${\bf L}_f=(L_{j,k})_{1\le j,k\le 2d-2}$,
where 
$$L_{j,k}=\lim_{l\to \infty}\frac{\frac{\partial f^l(c_j)}{\partial x_k}}
{(f^{l-1})'(v_j)},$$
is non-degenerate.
The rest of the proof is a straightforward application of this fact.
For every $l$ big enough, 
we introduce the following objects:

(1) an approximation matrix
${\bf L}^{(l)}_f=(L^{(l)}_{j,k})_{1\le j,k\le 2d-2}$,
where 
$$L^{(l)}_{j,k}=\frac{\frac{\partial f^l(c_j)}{\partial x_k}}{(f^{l-1})'(v_j)},$$

(2) two maps from the neighborhood $W$ of $\lambda_0$ into 
${\bf C}^{2d-2}$:
the map $C_l(\lambda)=\{f^l_\lambda(c_j(\lambda))\}_{1\le j\le 2d-2}$,
and the map $\bar a_l(\lambda)=\{a_j(\lambda)\}_{1\le j\le 2d-2}$
(recall that $z(\lambda)$ is the holomorphic motion of the point
$z=z(\lambda_0)\in K$, and
$a_j=f^l(c_j)\in K$),

(3) a diagonal $2d-2\times 2d-2$ matrix 
$D_l$ with the diagonal $(j,j)$ element
$(f^{l-1})'(v_j)$, $1\le j\le 2d-2$. 
Note that $(f^{l-1})'(v_j)\to \infty$ as $l\to \infty$. 

Since the family of analytic functions
$\{z(\lambda)\}$, where $\lambda\in W$,
is bounded, then the norms of the derivative $\bar a'_l$
of $\bar a_l(\lambda)$ at $\lambda_0$
are uniformly (in $l$) bounded, too.
By the definition, ${\bf h}_l(\lambda)=C_l(\lambda)-\bar a_l(\lambda)$, 
hence, for the derivatives at $\lambda=\lambda_0$,
${\bf h}_l'=C_l'-\bar a'_l$. 
By linear algebra,
we have a representation:
$$C_l'=D_l {\bf L}^{(l)}_f \pi'_0.$$
Therefore,
$${\bf h}_l'=D_l {\bf L}^{(l)}_f \pi'_0 - \bar a_l'=
D_l \left({\bf L}^{(l)}_f \pi'_0 - D_l^{-1} \bar a_l' \right).$$
As $l$ tends to $\infty$, 
${\bf L}^{(l)}_f$ tends
to the invertible matrix ${\bf L}_f$, and $D_l^{-1}$ tends to 
zero matrix. It follows, for every $l$ big enough, ${\bf h_l}'$ is invertible.
Moreover,
asymptotically, $({\bf h}_l')^{-1}$ is  
$(\pi'_0)^{-1} {\bf L}_f^{-1} D_l^{-1}$. This ends the proof.

\

{\bf Proof of Corollary~\ref{cormak}}

Let us first consider the case when $f$ is exceptional~\cite{dht}.
Since all critical points of $f$ are simple,
then $X_f$ is the (local)
space of all rational functions ${\bf Rat}_d$ of degree $d$ which are
close to $f$, and $\Lambda_f$ is its $2d-2$-dimensional submanifold.
As $f$ has no neutral cycles, by Proposition~\ref{gen} (b),
we have a well-defined map $g\in X_f\mapsto g_1\in \Lambda_f$,
where $g$ is equivalent to $g_1$,
and the Mobius conjugacy between them tends to the identity
uniformly on the Riemann sphere as $g\to f$ 
(see the end of Subsection~\ref{disc}).
It implies that $f$ is structurally stable in $X_f$ if and only if
$f$ is structurally stable in $\Lambda_f$. 
On the other hand, as $f$ is critically finite, it is easy to see that 
the critical relations of $f$ cannot be preserved for
all maps in $X_f={\bf Rat}_d$ which are close to $f$. Hence, 
$f$ is not structurally stable in $X_f$ and, then, in $\Lambda_f$.

Now, let $f$ be a non-exceptional map.
One can assume that $\omega(c)$ is a subset of the complex plane.
By applying Theorem~\ref{tworatintro} to $f$ and $c$, there exists
a one-dimentional submanifold $\Lambda$
(i.e., an analytic one-dimensional family $f_t$, $f_0=f$) of $\Lambda_f$,
such that the limit 
\begin{equation}\label{familyratcor}
\lim_{m\to \infty}\frac{\frac{d}{dt}|_{t=0}f_t^m (c(t))}{(f^{m-1})'(f(c))}
\end{equation}
exists and is a non-zero number.
(Here $c(t)$ is the critical point of $f_t$ which is close to
$c(0)=c$.)
This obviously implies that
\begin{equation}\label{notnormcor}
\lim_{m\to \infty}\frac{d}{dt}|_{t=0}f_t^m (c(t))=\infty.
\end{equation}
This is enough to conclude that $f$ is not structurally stable in
the family $f_t$, hence, in $\Lambda_f$, too.
Indeed,
assuming, by a contradiction, that $f$ is structurally stable in $f_t$,
we get that the $\omega$-limit set $\omega(c(t))$
of the critical point $c(t)$ of $f_t$ 
tends to $\omega(c(0))=\omega(c)$ as $t\to 0$
while the compact $\omega(c)$ lies in the plane.
In particular, the sequence of holomorphic functions
$\{f_t^m (c(t))\}_{m\ge 0}$ is uniformly bounded
in a neigborhood of $t=0$, which is a contradiction with~(\ref{notnormcor}). 
This ends the proof.

\section{Part (a) of Theorem~\ref{onerat}}
As in the polynomial case, we start by calculating partial derivatives 
of a function $g\in \Lambda_{d, \bar p'}$ 
w.r.t. the standard local coordinates
of the space $\Lambda_{d, \bar p'}$, 
i.e., $\sigma$, $b$, and the critical values
which are {\it not} infinity.
\begin{prop}\label{derlemrat}
Fix a function $f\in \Lambda_{d, \bar p'}$, where 
$f(z)=\sigma z + b + P(z)/Q(z)$.

(a) Let $v_k=f(c_k)$ be a critical value of $f$ which is
different from infinity (i.e., $1\le k\le p$). 
Then
$\frac{\partial f}{\partial v_k}(z)$ is a rational function $q_k(z)$ 
of degree $2d-2$ be of the form $\frac{\tilde P(z)}{(Q(z))^2}$,
where $\tilde P$ is a polynomial of degree at most $2d-3$.
It is uniquely characterized by the
following conditions: 

(i) at $z=c_k$, the function
$q_k(z)-1$ has zero of order
at least $m_k$; 

(ii) at $z=c_j$, for every $j\not=k$, $1\le j\le p$,
$q_k(z)$ has zero of order at least $m_j$; 

(iii) finally,
at $z=c_j$, for every $p+1\le j\le p'$,
the polynomial $\tilde P(z)$ has zero of order at least $m_j$.

In particular, if $c_k$ is simple (i.e., $m_k=1$),
then~(\ref{dersimp}) holds:
\begin{equation}\label{dersimprat}
\frac{\partial f}{\partial v_k}(z)=\frac{f'(z)}{f''(c_k)(z-c_k)}.
\end{equation}

(b) We have, as well:
\begin{equation}\label{sigma}
\frac{\partial f}{\partial \sigma}(z)=\frac{z}{\sigma}f'(z), \ \ \ \ 
\frac{\partial f}{\partial b}(z)=\frac{1}{\sigma}f'(z).
\end{equation}
\end{prop}
\begin{proof}
Let $x$ be one of the following coordinates of the vector $\bar v(g)$, for
$g\in \Lambda_{d, \bar p}$:
$x$ is either $v_k$, for $1\le k\le p$, or
$\sigma$, or $b$. Notice that, for $x=v_k$:
\begin{equation}\label{xvk}
\frac{\partial v_j(g)}{\partial x}=\delta_{j,k}, \ \ \ 1\le j\le p,
\end{equation}
and, for $x=\sigma$ or $x=b$:
\begin{equation}\label{xother}
\frac{\partial v_j(g)}{\partial x}=0, \ \ \ 1\le j\le p.
\end{equation} 
As $g\in \Lambda_{d, \bar p'}$ is close to $f$, then
$g(z)=\sigma(g)z+b(g)+\frac{P_g(z)}{Q_g(z)}$,
where $P_g(z)=B_0(g)z^{d-2}+...$, $Q_g(z)=z^{d-1}+...$
are polynomials with the coefficients which are holomorphic
functions of the vector $\bar v(g)$, and
$P_f=P$, $Q_f=Q$. 
It follows, $\frac{f}{\partial v_k}=\frac{\tilde P}{Q^2}$, where
$\tilde P=
\frac{\partial P}{\partial v_k} Q-P \frac{\partial Q}{\partial v_k}$
is a polynomial of degree at most $2d-3$.
We apply Proposition~\ref{dergenlem} to the coordinate $x$ as above
and take into account~(\ref{xvk})-(\ref{xother}). 
We conclude that,
for every $1\le j\le p$, $c_j$ is at least a $m_j$-multiple
zero of the functions
$\frac{\partial f}{\partial v_k}(z)-\delta_{j,k}$,
$\frac{\partial f}{\partial \sigma}(z)$, and 
$\frac{\partial f}{\partial b}(z)$.
In particular, all this proves the part (a) except for the last
property (iii).

Let $\hat P(z)=(\sigma z+b)Q(z)+P(z)$, that is,
$f=\frac{\hat P}{Q}$. Then
$\frac{\partial f}{\partial x}(z)$
is a rational function of the form   
$\frac{P^x(z)}{(Q(z))^2}$,
where 
\begin{equation}\label{tildep}
P^x=\frac{\partial \hat P}{\partial x} Q-\hat P \frac{\partial Q}{\partial x}.
\end{equation}
Now, let $p+1\le j\le p'$, i.e., $c_j$ is a root of $Q(z)$,
such that $Q(z)=(z-c_j)^{m_j+1}\psi(z)$, where $\psi$ is analytic near
$c_j$ and $\psi(c_j)\not=0$.
We cannot apply Proposition~\ref{dergenlem} directly in this case,
because $v_j=f(c_j)=\infty$.
To get around this, let us introduce a new (local) space
$N^*$ consisting of the functions $g^*(z):=1/g(z)$,
for $g\in \Lambda_{d, \bar p'}$ near $f$.  
Notice that the vector $\bar v(g)$ of coordinates in $\Lambda_{d, \bar p'}$
is also a local
coordinate in $N^*$, the critical points of $g^*$ and $g$
(with the corresponding multiplicities) coincide while for
the critical values $g^*$ and $g$, we have:
$v_j^*(g^*)=1/v_j(g)$. In particular,
$v_j^*=f^*(c_j)=0$.
Using the above notations, we have:
\begin{equation}\label{hatn}
\frac{\partial f^*}{\partial x}(z)=
-\frac{\frac{\partial f}{\partial x}(z)}{(f(z))^2}=
-\frac{P^x(z)}{(\hat P(z))^2}.
\end{equation}
On the other hand, we can apply Proposition~\ref{dergenlem}
to the space $N^*$, the map $f^*$, its critical point $c_j$
(with the corresponding critical value to be $0$),
and the coordinate $x$. Since $p<j\le p'$,
the critical value $v_j^*(g^*)=1/v_j(g)$
of $g^*$ is just the $j$-coordinate
of the vector of coordinates $\bar v(g)$ in $N^*$, while
$x$ (which is either $v_k(g)$, for
$1\le j\le p$, or $\sigma(g)$, or $b(g)$) is a different coordinate
of $\bar v(g)$. Hence, we also have:
$$\frac{\partial v_j^*}{\partial x}=0.$$
Then, by Proposition~\ref{dergenlem}, $c_j$ is at least $m_j$-multiple
root of $\frac{\partial f^*}{\partial x}(z)$, hence,
by~(\ref{hatn}), $c_j$ is at least $m_j$-multiple
root of $P^x(z)$, as well.
In particular, this holds for $\tilde P=P^{v_k}$, which proves 
(iii) of part (a).

As for part (b), $x$ is either $\sigma$ or $b$. Let $x=\sigma$.
Then
$$\frac{\partial f}{\partial \sigma}(z)=
z+\frac{R(z)}{(Q(z))^2}=\frac{P^\sigma(z)}{(Q(z))^2},$$
where 
$$R=\frac{\partial P}{\partial \sigma} Q-
P \frac{\partial Q}{\partial \sigma}$$
is a polynomial of degree at most $2d-3$, and $P^\sigma=z Q^2+R$.
In particular, $\frac{\partial f}{\partial \sigma}(z)=
z+O(1/z)$. On the other hand,
as we have seen, for every $1\le j\le p$, $c_j$ is at least a $m_j$-multiple
zero of $\frac{\partial f}{\partial \sigma}(z)$,
and, for $p<j\le p'$, $c_j$ is at least a $m_j$-multiple zero of $P^\sigma$.
Therefore, $\frac{\partial f}{\partial \sigma}(z)/f'(z)=
(z+O(1/z))/(\sigma +O(1/z^2))$, and, at the same time,
it is a rational function without poles in the plane.
Hence, it must be equal to $z/\sigma$.
The proof for $\partial f/\partial b$ is very similar
and is left to the reader.
\end{proof}
As in the polynomial case, we then have:
\begin{prop}\label{paramrat}
Let $f\in \Lambda_{d,\bar p'}$. Then
\begin{equation}\label{parameqrat}
T_f\frac{1}{z-x}=
\frac{1}{f'(z)}\frac{1}{f(z)-x}+
\sum_{k=1}^{p}\frac{\frac{\partial f}{\partial v_k}(z)}{f'(z)}
\frac{1}{x-v_k}.
\end{equation}
\end{prop}
\begin{proof}
We use Lemma~\ref{formula} and the part (a) of  
Proposition~\ref{derlemrat}.
By the assertions (1)-(2) of Lemma~\ref{formula}, $f'(z) L_k(z)$  
is a rational function of the form
$\tilde P/Q^2$, where $\tilde P$ is a polynomial
of degree at most $2d-3$,
such that  $f'(z) L_k(z)-1$ has root
at $c_k$ with multiplicity at least $m_k$, and, for every $j\not=k$,
$1\le j\le p'$, the polynomial $\tilde P$ has root at $c_j$
with multiplicity at least $m_j$.
Hence, $f'(z) L_k(z)$ coincides with the rational function $q_k(z)$.
Therefore, indeed, 
$L_k(z)=\frac{\frac{\partial f}{\partial v_k}(z)}{f'(z)}$.
\end{proof}
If $v_k=\infty$ or $z=\infty$, 
Proposition~\ref{derlemrat} is not useful.
In that case, we replace the space
$\Lambda_{d, \bar p'}$ by the space
$$\Lambda^M_{d, \bar p'}=\{M^{-1}\circ g\circ M: g\in 
\Lambda_{d, \bar p'}\},$$ 
where $M$ is a non-linear Mobius transformation,
that is, 
$\beta=M^{-1}(\infty)\not=\infty$, $\alpha=M(\infty)\not=\infty$,
and $\alpha$ is such that
the forward orbits of critical points of $f$ are
different from $\alpha$. Then
the forward critical orbits
of $\tilde f=M^{-1}\circ f\circ M$ lie in the complex plane.
Hence, for any $l>0$, if $g$ is close enough to $f$, then
the first $l$ iterates of the critical points of $\tilde g$ 
lie in the plane, too.

By Definition~\ref{Mspace},
$\tilde g=M^{-1}\circ g\circ M\in \Lambda_{d, \bar p'}^M$ 
is a holomorphic function of 
$\sigma(g)$, $b(g)$, and the critical
values $v_j(\tilde g)=M^{-1}(v_j(g))$ of $\tilde g$, $1\le j\le p'$.
By $\tilde v_j$, $\tilde c_j$ we denote $v_j(\tilde f)$, 
$c_j(\tilde f)$, i.e., the critical values and points of
$\tilde f$. In particular, $\tilde v_j=\beta$ if and only if $p+1\le j\le p'$.
As usual, $\{v_k\}$ denote the critical values of $f$,
and $\partial \tilde f^n/\partial \tilde v_k$ means 
$\partial \tilde g^n/\partial v_k(\tilde g)$ calculated at 
$\tilde g=\tilde f$. 
\begin{prop}\label{derlemratM}
(a) Take $l\ge 1$. If $f^i(v_j)\not=\infty$
for $0\le i\le l-1$, then
\begin{equation}\label{conM}
\frac{\frac{\partial \tilde f^l}{\partial \tilde v_k}(\tilde c_j)}
{(\tilde f^{l-1})'(\tilde v_j)}=\frac{(M^{-1})'(v_j)}{(M^{-1})'(v_k)}
\frac{\frac{\partial f^l}{\partial v_k}(c_j)}{(f^{l-1})'(v_j)},
\end{equation}
\begin{equation}\label{conMsigma}
\frac{\frac{\partial \tilde f^l}{\partial \sigma}(\tilde c_j)}
{(\tilde f^{l-1})'(\tilde v_j)}=
(M^{-1})'(v_j)
\frac{\frac{\partial f^l}{\partial \sigma}(c_j)}
{(f^{l-1})'(v_j)}, 
\end{equation}
\begin{equation}\label{conMb}
\frac{\frac{\partial \tilde f^l}{\partial b}(\tilde c_j)}
{(\tilde f^{l-1})'(\tilde v_j)}=
(M^{-1})'(v_j)
\frac{\frac{\partial f^l}{\partial b}(c_j)}
{(f^{l-1})'(v_j)}
\end{equation}
(b) We have: 
\begin{equation}\label{conM0}
\frac{\partial \tilde f}{\partial \tilde v_k}(\beta)=0, \ \ 1\le k\le p', 
\ \ \
\frac{\partial \tilde f}{\partial \sigma}(\beta)=
\frac{\partial \tilde f}{\partial b}(\beta)=0,
\end{equation}
\begin{equation}\label{conMcr}
\frac{\partial \tilde f}{\partial \tilde v_k}(\tilde c_j)=\delta_{j,k}, \ \ 
\frac{\partial \tilde f}{\partial \sigma}(\tilde c_j)
=\frac{\partial \tilde f}{\partial b}(\tilde c_j)=0,
\ \ 1\le k,j\le p'.
\end{equation}
\end{prop}
\begin{proof}
(a) Since $\tilde g^l=M^{-1}\circ g^l\circ M$, 
$v_k(g)=M(v_k(\tilde g))$,
where $M$ is a fixed map, and
by the conditions on $v_k, c_j$, the following calculations make sense:
\begin{eqnarray*}
\frac{\frac{\partial \tilde f^l}{\partial \tilde v_k}(\tilde c_j)}
{(\tilde f^{l-1})'(\tilde v_j)} &=& \\
\frac{(M^{-1})'(f^l(M(\tilde c_j))\frac{\partial f^l}{\partial v_k}
(M(\tilde c_j))\frac{\partial v_k(g)}{\partial \tilde v_k}}
{(M^{-1})'(f^{l-1}(M(\tilde v_j))(f^{l-1})'(M(\tilde v_j))M'(\tilde v_j)} &=& \\
\frac{(M^{-1})'(f^{l-1}(v_j))\frac{\partial f^l}{\partial v_k}
(c_j)M'(\tilde v_k)}
{(M^{-1})'(f^{l-1}(v_j)(f^{l-1})'(v_j)M'(\tilde v_j)} &=& \\
\frac{(M^{-1})'(v_j)}{(M^{-1})'(v_k)}
\frac{\frac{\partial f^l}{\partial v_k}(c_j)}
{(f^{l-1})'(v_j)}.
\end{eqnarray*}
The proof of~(\ref{conMsigma})-(\ref{conMb}) is similar.

(b) Since $\tilde g(\beta)=\beta\not=\infty$ 
for every $\tilde g\in \Lambda^M_{d, \bar p'}$,  
we write:
\begin{equation}\label{1M}
\tilde g(z)=\beta+\int_{\beta}^z \tilde g'(w) dw.
\end{equation}
Since $\beta$ is independent of $g$, 
\begin{equation}\label{2M}
\frac{\partial \tilde f}{\partial \tilde v_k}(z)=
\int_{\beta}^z \frac{\partial \tilde f'}{\partial \tilde v_k}(w) dw,
\end{equation}
and similar for the derivatives w.r.t. $\sigma$ and $b$.
To show~(\ref{conM0}), it remains to estimate~(\ref{2M}) at $z=\beta$.
To show~(\ref{conMcr}),
it is enough to apply 
Proposition~\ref{dergenlem} to the space $\Lambda_{d, \bar p}^M$,
the map $\tilde f$, the critical point $\tilde c_j$,
and the coordinate in the vector $\bar v^M$, which is either $\tilde v_k$
or $\sigma$ or $b$.
\end{proof}

{\bf Proof of Part (a) of Theorem~\ref{onerat}.}

Here we state and prove a refined version of Theorem~\ref{onerat}(a)
introducing notations $L(c_j, v_k)$, $L(c_j, \sigma)$,
$L(c_j, b)$ to be used later on. Recall that the notations
$L^M(c_j, v_k)$, $L^M(c_j, \sigma)$, $L^M(c_j, b)$ are introduced
in Theorem~\ref{onerat}. 
\begin{theo}\label{onerata}
Let $f\in \Lambda_{d, \bar p'}$.

(1) Assume $v_k\not=\infty$, $c_j$ is summable,
and the whole orbit of $c_j$ remains in the complex plane.
Then the limits defined below exist and are expressed as the following
series, which converge absolutely:
\begin{equation}\label{limrata}
L(c_j, v_k):=\lim_{l\to \infty}
\frac{\frac{\partial (f^l (c))}{\partial v_k}}{(f^{l-1})'(v)}=
\delta_{j,k}+\sum_{n=1}^{\infty}
\frac{\frac{\partial f}{\partial v_k}(f^{n-1}(v_j))}{(f^{n})'(v_j)},
\end{equation}
\begin{equation}\label{limsigmaa}
L(c_j, \sigma):=\lim_{l\to \infty}
\frac{\frac{\partial (f^l (c_j))}{\partial \sigma}}{(f^{l-1})'(v_j)}=
\frac{1}{\sigma}\sum_{n=0}^\infty \frac{f^n(v_j)}{(f^n)'(v_j)},
\end{equation}
\begin{equation}\label{limma}
L(c_j, b):=\lim_{l\to \infty}
\frac{\frac{\partial (f^l (c_j))}{\partial b}}{(f^{l-1})'(v_j)}=
\frac{1}{\sigma}
\sum_{n=0}^\infty \frac{1}{(f^n)'(v)}.
\end{equation}
Furthermore,
\begin{equation}\label{lmlv}
L^M(c_j, v_k)=\frac{(M^{-1})'(v_j)}{(M^{-1})'(v_k)}L(c_j, v_k),
\end{equation}
\begin{equation}\label{lmlsigma}
L^M(c_j, \sigma)=(M^{-1})'(v_j)L(c_j, \sigma), \ \ \ 
L^M(c_j, b)=(M^{-1})'(v_j)L(c_j, b).
\end{equation}

(2) Assume $v_k\not=\infty$, $c_j$ is summable 
and $f^l(v_j)=\infty$, for some minimal $l\ge 1$.
Define in this case:
\begin{equation}\label{limratinf}
L(c_j, v_k):=
\delta_{j,k}+\sum_{n=1}^{l}
\frac{\frac{\partial f}{\partial v_k}(f^{n-1}(v_j))}{(f^{n})'(v_j)},
\end{equation}
\begin{equation}\label{limsigmainf}
L(c_j, \sigma):=
\frac{1}{\sigma}\sum_{n=0}^{l-1} \frac{f^n(v_j)}{(f^n)'(v_j)}, \ \
L(c_j, b):=
\frac{1}{\sigma}
\sum_{n=0}^{l-1} \frac{1}{(f^n)'(v)}.
\end{equation}
Then~(\ref{lmlv})-(\ref{lmlsigma}) 
hold.

(3) Assume $v_j=f(c_j)=\infty$. Then, for 
$1\le k\le p'$,
\begin{equation}\label{limratMinfa}
L^M(c_j, v_k)=\delta_{j,k}, 
\end{equation}
\begin{equation}\label{limratMinf1a}
L^M(c_j, \sigma)=L^M(c_j, b)=0.
\end{equation}
\end{theo}
\begin{proof}
(1) It is enough to prove 
(\ref{limrata})-(\ref{limma}). Then,
by~(\ref{conM})-(\ref{conMb}) of Proposition~\ref{derlemratM},
~(\ref{lmlv})-(\ref{lmlsigma}) follow.
We can use the identity (\ref{6}) of Subsection~\ref{pra}: 
for every $l>0$,
\begin{equation}\label{nuzhnoa}
\frac{\frac{\partial f^l}{\partial v_k}(c_j)}
{(f^{l-1})'(v_j)}=
\frac{\partial f}{\partial v_k}(c_j)+
\sum_{n=1}^{l-1}\frac{\frac{\partial f}{\partial v_k}(f^{n-1}(v_j))}
{(f^{n})'(v_j)}.
\end{equation}
By the part (a) of Proposition~\ref{derlemrat}, 
$\frac{\partial f}{\partial v_k}(c_j)=\delta_{j,k}$, and, moreover,
$\frac{\partial f}{\partial v_k}(z)$
is a rational function of the form $\tilde P/Q^2$, where
degree of $\tilde P$ 
is less than the degree of $Q^2$. 
Hence, for some constant $C_k$ and all $z$,
\begin{equation}\label{bddrat}
|\frac{\partial f}{\partial v_k}(z)|\le C_k (1+|f(z)|^{2}).
\end{equation}
Now, assume that $c_j$ is weakly expanding.
Then
\begin{equation}\label{aaarat}
\sum_{n=1}^{\infty}
\frac{|\frac{\partial f}{\partial v_k}(f^{n-1}(v_j))|}{|(f^{n})'(v_j)|}
\le C_k
\frac{1+|f^{n}(v_j)|^2}{|(f^{n})'(v_j)|}<\infty,
\end{equation}
and~(\ref{limrata}) follows.
The proof of existence of $L(c_j, \sigma)$ and $L(c_j, b)$
is similar to the proof for $v_k$-derivative. Indeed,
$$
\frac{\partial f^l}{\partial \sigma}(z)=
(f^l)'(z)\sum_{n=0}^{l-1}\frac{\frac{\partial f}{\partial \sigma}(f^n(z))}
{(f^{n+1})'(z)}.$$
Letting here $z\to c_j$ and using the part (b) 
of Proposition~\ref{derlemrat}, 
one gets:
\begin{eqnarray*}
\frac{\partial f^l}{\partial \sigma}(c_j)=
(f^{l-1})'(v_j)\left\{\frac{\partial f}{\partial \sigma}(c_j)+
\sum_{n=1}^{l-1}\frac{\frac{\partial f}{\partial \sigma}(f^{n-1}(v_j))}
{(f^{n})'(v_j)}\right\} &=& \\
(f^{l-1})'(v_j)\frac{1}{\sigma}\left\{c_j f'(c_j)+
\sum_{n=1}^{l-1}\frac{f^{n-1}(v_j) f'(f^{n-1}(v_j))}
{(f^{n})'(v_j)}\right\} &=& \\
\frac{(f^{l-1})'(v_j)}{\sigma}\sum_{n=1}^{l-1}
\frac{f^{n-1}(v_j)}
{(f^{n-1})'(v_j)},
\end{eqnarray*}
and we get~(\ref{limsigmaa}).
Doing the same (with obvious changes) for the
$\partial/\partial b$-derivative, we get~(\ref{limma}).

(2) We use (\ref{nuzhnoa}) with $\tilde f$ instead of $f$.
Note that $f^j(v_j)=\infty$ if and only if
$\tilde f^j(\tilde v_j)=\beta$. Then $\tilde f^j(\tilde v_j)=\beta$
for every $j\ge l$, and, hence, for $j\ge l$, 
by Proposition~\ref{derlemratM} (b),
\begin{eqnarray*}
\frac{\frac{\partial f^j}{\partial \tilde v_k}(\tilde c_j)}
{(\tilde f^{j-1})'(\tilde v_j)}=
\frac{\partial \tilde f}{\partial \tilde v_k}(\tilde c_j)+
\sum_{n=1}^{j-1}\frac{\frac{\partial \tilde f}{\partial \tilde v_k}
(\tilde f^{n-1}(\tilde v_j))}
{(\tilde f^{n})'(\tilde v_j)} &=& \\
\delta_{j, k}+
\sum_{n=1}^{l-1}\frac{\frac{\partial \tilde f}{\partial \tilde v_k}
(\tilde f^{n-1}(\tilde v_j))}
{(\tilde f^{n})'(\tilde v_j)}.
\end{eqnarray*}
That is,
\begin{equation}\label{nuzhnoaa}
\frac{\frac{\partial f^j}{\partial \tilde v_k}(\tilde c_j)}
{(\tilde f^{j-1})'(\tilde v_j)}
=\frac{\frac{\partial f^l}{\partial \tilde v_k}(\tilde c_j)}
{(\tilde f^{l-1})'(\tilde v_j)},
\end{equation}
while, by Proposition~\ref{derlemratM} (a),
$$\frac{\frac{\partial f^l}{\partial \tilde v_k}(\tilde c_j)}
{(\tilde f^{l-1})'(\tilde v_j)}
=\frac{(M^{-1})'(v_j)}{(M^{-1})'(v_k)}
\frac{\frac{\partial f^l}{\partial v_k}(c_j)}{(f^{l-1})'(v_j)}=
\frac{(M^{-1})'(v_j)}{(M^{-1})'(v_k)}L(c_j, v_k).$$
This proves~(\ref{lmlv}) in the considered case.
The proof of~(\ref{lmlsigma}) is similar.

(3). The relation~(\ref{limratMinfa}) follows directly from~(\ref{nuzhnoaa}) 
(with $l=1$) and Proposition~\ref{derlemratM}. 
The proof of~(\ref{limratMinf1a}) is similar.
\end{proof}

Next statement is a corollary from the previous Theorem~\ref{onerata}.
It allows us to reduce evaluation of the rank of the
matrix ${\bf L^M}$, which is defined in Theorem~\ref{onerat} (b),
to the evaluation of the rank of a simpler matrix ${\bf L}$ defined below. 
Recall that the integer number $\nu$ between $0$ and $r$
is defined in Theorem~\ref{onerat} (b) as the maximal number
of points from the collection
$c_{j_1},...c_{j_r}$, $1\le j_1<...<j_r\le p'$, 
of $r$ summable critical points of $f$, such that
the corresponding critical values are different from infinity. Moreover,
if $\nu<r$, the last $r-\nu$ indexes $j_{\nu+1},...,j_r$
are $p'-(r-\nu-1),...,p'$. 
\begin{coro}\label{lmtol}
The rank of the matrix ${\bf L^M}$ is bigger than or equal to 
\begin{equation}\label{rank}
r':=(r-\nu)+ r_0, 
\end{equation}
where $r_0\le \nu$ is  
the rank of the matrix ${\bf L}$ obtained from ${\bf L^M}$
by the following three operations:

(1) firstly, if $\nu<r$, then we cross out the last $r-\nu$ rows 
and the last $r-\nu$ columns of 
${\bf L^M}$ (if $\nu=r$, we do nothing); let ${\bf L^M_1}$ be the resulting
matrix,

(2) secondly, we cross out the last $(p'-p)-(r-\nu)$
columns of ${\bf L^M_1}$ (if $p'=p$, we again do nothing);
let the resulting matrix be ${\bf L^M_2}$,

(3) at last, we replace in ${\bf L^M_2}$ the elements
$L^M(c_{j_i}, v_k)$, $L^M(c_{j_i}, \sigma)$ 
by $L(c_{j_i}, v_k)$,  $L(c_{j_i}, \sigma)$ respectively;
the resulting matrix is called ${\bf L}$.

In other words, in the case
($H_\infty$),
\begin{equation}\label{rmdeltacut}
{\bf L}=(L(c_{j_i}, \sigma), L(c_{j_i}, v_1),..., L(c_{j_i}, v_{p-1}))_{1\le i\le \nu},
\end{equation}
In the case
($NN_\infty$),
\begin{equation}\label{rmdeltacutn}
{\bf L}=(L(c_{j_i}, v_1),..., L(c_{j_i}, v_{p-1}))_{1\le i\le \nu},
\end{equation}
and in the case
($ND_\infty$),
\begin{equation}\label{rmdeltacutnn}
{\bf L}=(L(c_{j_i}, v_1),...,L(c_{j_i}, v_{p-2})_{1\le i\le \nu}.
\end{equation}
\end{coro}
\begin{proof}
Let us consider last $r-\nu$ rows of ${\bf L^M}$,
which correspond to indexes $j_i=p'-r+i$
of the summable critical points, for $i=\nu+1,...,r$.
By Theorem~\ref{onerata} (3), 
on the $i$-row of ${\bf L^M}$, for
$i=\nu+1,...,r$, 
all elements are $0$,
except for the ''diagonal'' one, which is equal to $1$:
$L^M(c_{p'-r+i}, v_{p'-r+i})=1$.
In other words, the lower right hand corner of the matrix
${\bf L^M}$ contains the $r-\nu\times r-\nu$ identity matrix.
Therefore, the rank of ${\bf L^M}$ is bigger than or equal to $r-\nu$ plus
the rank of ${\bf L^M_1}$.
If we cross out some of the columns of ${\bf L^M_1}$, then the rank
can only decrease. Hence, the rank of ${\bf L^M_1}$ is bigger than or equal
to the rank of ${\bf L^M_2}$. In turn, its elements are
the corresponding elements of the matrix ${\bf L}$ defined above,
but with the upper index $M$.
Now, we employ the relations~(\ref{lmlv})-(\ref{lmlsigma})
of Theorem~\ref{onerata} and conclude that
the ranks of ${\bf L^M_2}$ and ${\bf L}$ are actually equal.
\end{proof}
\section{Part (b) of Theorem~\ref{onerat}}
\subsection{The Wolff-Denjoy series}\label{wd}
Let $\{\alpha_k\}_{k=0}^\infty$, 
$\{b_k\}_{k=0}^\infty$ (where $b_i\not=b_j$, for $k\not=j$)
be two sequences of complex numbers.
Assume that 
\begin{equation}\label{ab1}
\sum_{k\ge 0} |\alpha_k|(1+|b_k|^2)<\infty,
\end{equation}
It implies that
two series
\begin{equation}\label{ab2}
A=\sum_{k\ge 0} \alpha_k, \ \ \
B=\sum_{k\ge 0} \alpha_k b_k,
\end{equation}
converge absolutely.
Define
\begin{equation}\label{H}
H(x)=\sum_{k=0}^\infty \frac{\alpha_k}{b_k-x}.
\end{equation}
Then $H$ is integrable in every disk $B_r=\{|x|<r\}$.
Indeed,
if $\sigma_x$ denotes the element of the Lebesgue measure on the plane
of the variable $x$, then
$$\int_{B_r} |H(x)| d\sigma_x\le
\sum_{k\ge 0}|\alpha_k|\int_{B_r}\frac{1}{|x-b_k|} d\sigma_x\le 
2\pi \sum_{k\ge 0}|\alpha_k|(r+|b_k|)<\infty.$$
We define a kind of regularization of $H$ at infinity
as
$$\hat H(x)=H(x)+\frac{A}{x}+
\frac{B}{x^2}=
\sum_{n\ge 0}(\frac{\alpha_n}{b_n-x}+\frac{\alpha_n}{x}+
\frac{\alpha_n b_n}{x^2}).$$
The name is justified by the following claim
(which must be known though):
\begin{lem}\label{int}
$\hat H$ is integrable at infinity.
\end{lem}
\begin{proof}
For every $n$, the function 
$1/(b_n-x)+1/x+b_n/x^2$ is integrable at infinity, and one can write,
for $r>0$:
$$\int_{|x|>r} \left|\frac{1}{b_n-x}+\frac{1}{x}+\frac{b_n}{x^2}\right|d\sigma_x
=|b_n|\int_{|x|>r/|b_n|} \left|\frac{1}{1-x}+\frac{1}{x}+\frac{1}{x^2}\right|d\sigma_x\le $$
$$|b_n|\left(C_1+\int_{2>|x|>r/|b_n|} \frac{1}{|x|^2}d\sigma_x\right)
\le |b_n|\left(C_2+2\pi \ln\frac{|b_n|}{r}\right)\le
C_3 |b_n|(1+|b_n|),$$
where the constants $C_i$ here and below depend only on $r$.
Now, for every $R$ big enough, and using the 
condition (\ref{ab1}), we have:
\begin{eqnarray*}
\int_{r<|x|<R}|\hat H(x)| d\sigma_x\le
\sum_{k=0}^\infty |\alpha_k|\int_{|x|>r}
\left|\frac{1}{b_k-x}+\frac{1}{x}+\frac{b_k}{x^2}\right|d\sigma_x\le
C_3\sum_{k=0}^\infty |\alpha_k||b_k|(1+|b_k|)\le C_4,
\end{eqnarray*}
where $C_4$ does not depend on $R$.
\end{proof}

We need a contraction property of the operator $T_f$.
This property is initially due to Thurston, see~\cite{dht}. 
The proof of the 
following claim is a minor variation of~\cite{dht},~\cite{e},
~\cite{mak},~\cite{mult}
\begin{lem}\label{contr}
Let $H$ be defined by the series~(\ref{H}), under the condition~(\ref{ab1}),
and the numbers $A, B$ are defined by~(\ref{ab2}).
Denote by $K$ the closure (on the Riemann sphere)
of the set $\{b_k\}$:
\begin{equation}\label{K}
K:=\overline{\{b_k: k\ge 0\}}\subset {\bf \bar C}.
\end{equation} 
Assume that $K$ has no interior points.
Let $f$ be a rational function with the asymptotics
at infinity $f(z)=\sigma z + b + O(1/z)$,
such that $H$ is a fixed point of the operator $T_f$
associated to $f$.

(1) If $A=B=0$, then 
either $H=0$ on the complement $K^c$ of $K$,
or $f$ is an exceptional map.

(2) If either $|\sigma|\ge 1$ and $b=0$, 
or $A=0$ and $\sigma=1$,
then $H=0$ on $K^c$, too.

\end{lem}
\begin{proof}
Note that $H$ is analytic in each component of $K^c$.
Now, take $R$ big enough and consider the disk $D(R)=\{|x|<R\}$. 
We claim that
\begin{equation}\label{limr2}
\lim_{R\to \infty} \left\{\int_{f^{-1}(D(R))} |H(x)|d\sigma_x-
\int_{D(R)} |H(x)|d\sigma_x\right\} \le 0.
\end{equation}
Indeed, in the case (1), this follows at once from
the integrability of $H$ at infinity.  
In the case (2),
the conditions on $\sigma$ imply that there is $a>0$, such that
\begin{equation}\label{f-1}
f^{-1}(D(R))\subset D(R+|b|+\frac{a}{R})
\end{equation}
(actually, $f^{-1}(D(R))\subset D(R)$, if $|\sigma|>1$).
On the other hand,
\begin{equation}\label{limr}
\lim_{R\to \infty} \int_{R<|x|<R+|b|+a/R} |H(x)|d\sigma_x=0.
\end{equation}
Indeed, by Lemma~\ref{int}, $H(x)=\hat H(x) - A/x - B/x^2$, where
$\hat H$ is integrable at infinity. In particular,
$\lim_{R\to \infty} \int_{R<|x|<R+|b|+a/R} |\hat H(x)|d\sigma_x=0$.
But an easy calculation shows that the conditions in the case (2)
guarantee that
\begin{equation}
\lim_{R\to \infty} \int_{R<|x|<R+|b|+a/R} \left|\frac{A}{x}+\frac{B}{x^2}\right|d\sigma_x=0,
\end{equation}
so (\ref{limr}) is proved. This, along with~(\ref{f-1}), gives us 
(\ref{limr2}) in the case (2).

Let us show that 
\begin{equation}\label{tri}
|T_fH(x)|=\sum_{w: f(w)=x}\frac{|H(w)|}{|f'(w)|^2}
\end{equation}
almost everywhere. Indeed, otherwise there is a 
set $A\subset D(R_0)$ of positive measure (for some $R_0$)
and $\delta>0$, such that 
$|T_fH(x)|<(1-\delta)\sum_{w: f(w)=x}\frac{|H(w)|}{|f'(w)|^2}$ for $x\in A$.
Then,
for all $R>R_0$,
\begin{eqnarray*}
\int_{D(R)} |H(x)|d\sigma_x=\int_{D(R)\setminus A} |H(x)|d\sigma_x
+\int_{A} |H(x)|d\sigma_x &=& \\
\int_{D(R)\setminus A} |T_fH(x)|d\sigma_x
+\int_{A} |T_fH(x)|d\sigma_x &<& \\
\int_{f^{-1}(D(R)\setminus A)} |H(x)|d\sigma_x+
(1-\delta)\int_{f^{-1}(A)} |H(x)|d\sigma_x &=& \\
\int_{f^{-1}(D(R))} |H(x)|d\sigma_x-
\delta \int_{f^{-1}(A)} |H(x)|d\sigma_x,
\end{eqnarray*}
which contradicts~(\ref{limr2}).
With~(\ref{tri}) holding almost everywhere, we proceed as
in the above cited papers.
\end{proof}
\begin{lem}\label{ccomp}
Let $\alpha_k, b_k, k\ge 0$ be two sequences of 
complex numbers, such that $\sum_{k\ge 0} |\alpha_k|<\infty$
and the closure $K$ on the Riemann sphere of the set
$\{b_k, k\ge 0\}$, where the points $b_k$
are pairwise different, is a C-compact $K$.
If $H(x)=\sum_{k\ge 0} \frac{\alpha_k}{b_k-x}$ is equal to $0$
outside of $K$, then $\alpha_k=0$ for every $k$.
\end{lem}
\begin{proof}
(1) The case when $K$ is a C-compact in the plane is classical,
see e.g.~\cite{bsz} (and also the proof of Theorem~\ref{one}).
(2) Now, assume that $\infty\in K$. Take $x_0\in {\bf C}\setminus K$.
Let $\epsilon>0$ be so that $|b_k-x_0|>\epsilon$ for every $k$.
Let $c_k=b_k-x_0$.
Then the function
$H_1(y)=\sum_{k\ge 0} \alpha_k/(c_k-y)$ is equal to $0$
outside of the compact $K_1=K-x_0$, which is also a C-compact,
but does not contain the origin.
By the definition, the compact $K_2=1/K_1=\{1/y: y\in K_1\}$
is a C-compact on the plane. 
Consider $H_2(z)=\sum_{k\ge 0} (\alpha_k/c_k)/(1/c_k-z)$.
Since $|c_k|>\epsilon$, still $\sum_{k\ge 0} |\alpha_k/c_k|<\infty$.
But, for every $z=1/y$ outside of $K_2$, so that $y$ is outside of $K_1$,   
$H_2(z)=-y H_1(y)=0$. Then we apply the case (1).
\end{proof}
\subsection{The operator identity for rational function.}
Since the proof of
Proposition~\ref{main} is formal, we get using Proposition~\ref{paramrat}:
\begin{prop}\label{mainrat}
Let $f\in \Lambda_{d, \bar p'}$.
Given $z\in {\bf C}$, we define $l(z)$ to be the minimal 
$l$, such that $f^l(z)=\infty$. If there is no such $l$, then $l(z)=\infty$.
We have:
\begin{equation}\label{oldrat}
\varphi_{z,\lambda}(x)-\lambda (T_f\varphi_{z,\lambda})(x)=\frac{1}{z-x}+
\lambda\sum_{k=1}^{p}\frac{1}{v_k-x}
\Phi_{k}(\lambda, z),
\end{equation}
where 
\begin{equation}\label{varphirat}
\varphi_{z,\lambda}(x)=\sum_{n=0}^{l(z)-1}\frac{\lambda^n}{(f^n)'(z)}
\frac{1}{f^n(z)-x},
\end{equation}
\begin{equation}\label{Psirat}
\Phi_{k}(\lambda, z)=\sum_{n=0}^{l(z)-1} \lambda^{n+1}
\frac{\frac{\partial f}{\partial v_k}(f^n(z))}{(f^{n+1})'(z)}.
\end{equation}
\end{prop}

Putting in Proposition~\ref{mainrat} $\lambda=1$ and $z=v_j$ and
combining it with the definition of $L(.,.)$ from Theorem~\ref{onerata}, we get:
\begin{prop}\label{vaznoerat}
Let $c_j$, $1\le j\le p$, be a summable critical point of $f\in \Lambda_{d, \bar p'}$,
such that $v_j\not=\infty$.
Then
\begin{equation}\label{vazeqrat}
H_j(x)-(T_f H_j)(x)=\sum_{k=1}^{p}\frac{L(c_j, v_k)}{v_k-x},
\end{equation}
where 
\begin{equation}\label{hjrat}
H_j(x)=\sum_{n=0}^{l(v_j)-1}\frac{1}{(f^n)'(v_j)(f^n(v_j)-x)},
\end{equation}
\end{prop}
\begin{com}
Note that the sum in~(\ref{vazeqrat}) is over the critical values
which are different from infinity.
In particular, $l(v_j)\ge 1$ for $1\le j\le p$.
\end{com}
\subsection{Proof of Part (b) of Theorem~\ref{onerat}}
By Corollary~\ref{lmtol}, it is enough to show that the rank of the
matrix ${\bf L}$ is equal to $\nu$.
Now, we follow closely the proof of Theorem 6 of~\cite{mult}.
Suppose first we are in the 
case ${\bf H_\infty}$. 
Assume the rank of the matrix ${\bf L}$ 
is less than $\nu$.
Denote by $S$ the set of indexes $j_1,...,j_{\nu}$.
Then the vectors 
\begin{equation}\label{vh}
(L(c_j, \sigma), L(c_j, v_1), L(c_j, v_{2}),..., L(c_j, v_{p-1})), \ \ \
j\in S,
\end{equation}
are linearly dependent: there exist $\nu$ numbers $a_j$, $j\in S$,
not all zeros, such that:
\begin{equation}\label{deprat}
\sum_{j\in S} a_j L(c_j, \sigma)=0, \ \ \ \sum_{j\in S} a_j L(c_j, v_k)=0,\ 1\le k\le p-1.
\end{equation}
On the other hand, for every $j\in S$,
\begin{equation}\label{vazeq1}
H_j(x)-(T_f H_j)(x)=\sum_{k=1}^{p}\frac{L(c_j, v_k)}{v_k-x},
\end{equation}
where 
\begin{equation}\label{hj1}
H_j(x)=\sum_{n=0}^{l(v_j)-1}\frac{1}{(f^n)'(v_j)(f^n(v_j)-x)},
\end{equation}
By~(\ref{deprat}), it gives us, that 
the linear
combination $H=\sum_{j\in S} a_j H_j$ satifies the relation:
\begin{equation}\label{rl}
H(z)-(T_f H)(z)=\frac{L}{v_p-z}.
\end{equation}
The function $H$ is reduced to the form
\begin{equation}\label{hformrat}
H(z)=\sum_{k=0}^\infty \frac{\alpha_k}{b_k-z}, \ \ \ b_k\not=b_j, 
\ \ k\not=j,
\end{equation}
and since each $c_j$, $j\in S$, is summable,
the sequences $\alpha_k$, $b_k$ satisfy the condition
(\ref{ab1}).
Recall the notations
$A=\sum_{k\ge 0}\alpha_k$,
$B=\sum_{k\ge 0} \alpha_k b_k$, and
the regularization $\hat H$ of $H$ is:
$$\hat H(z)=H(z)+\frac{A}{z}+\frac{B}{z^2}.$$
We use the following asymptotics, as $z\to \infty$,
which are easily checked
(see e.g. the proof of Lemma 8.2 of~\cite{mult}):
\begin{equation}\label{ass}
T_f\frac{1}{z}=\frac{1}{\sigma z}+\frac{b}{\sigma z^2}
+O(\frac{1}{z^3}), \ \ T_f\frac{1}{z^2}=\frac{1}{z^2}+ O(\frac{1}{z^3}).
\end{equation}
Now, we can rewrite~(\ref{rl}) as follows:
$$\hat H(z)-(T_f\hat H)(z)=\frac{A}{z}-
A(\frac{1}{\sigma z}+\frac{b}{\sigma z^2})
-\frac{L}{z}-\frac{L v_p}{z^2}+O(\frac{1}{z^3}),$$
or
\begin{equation}\label{h-th}
\hat H(z)-(T_f\hat H)(z)=\frac{A(1-1/\sigma)-L}{z}+
\frac{-A b/\sigma-L v_p}{z^2}+O(\frac{1}{z^3}).
\end{equation}
But the function $\hat H$ is integrable at infinity, hence, so is
$\hat H-T_f\hat H$, and we can write:
$$I_R:=\int_{|z|>R}|\hat H(z)-(T_f\hat H)(z)|d\sigma_z\le
\int_{|z|>R}|\hat H(z)| d\sigma_z+
\int_{f^{-1}({|z|>R})}|\hat H(z)| d\sigma_z.$$
It implies that $I_R\to 0$ as $R\to \infty$.
But then $\hat H(z)-T_f\hat H(z)=O(\frac{1}{z^3})$ at infinity. Hence,
~(\ref{h-th}) implies:
\begin{equation}\label{rlm}
A(1-\frac{1}{\sigma})-L=0, \ \ \ \ -\frac{A b}{\sigma}-L v_p=0.
\end{equation}
Since we are in the case ${\bf H_\infty}$, then $b=0$ and
$v_p=1$. Hence, by the second relation of~(\ref{rlm}), $L=0$.
That is, by~(\ref{rl}), $H$ is a fixed point
of $T_f$. Furthermore, by another condition of
the case ${\bf H_\infty}$, $\sigma\not=1$, which, along
with $L=0$ and the first relation of~(\ref{rlm}), imply that $A=0$.
Now we use that $\sum_{j\in S} a_j L(c_j, \sigma)=0$.
By assertions~(\ref{limsigmaa}),~(\ref{limsigmainf}) of 
Theorem~\ref{onerata},
\begin{equation}\label{b}
B=\sum_{k\ge 0} \alpha_k b_k=\sum_{j\in S}
a_j \sum_{n\ge 0} \frac{f^n(v_j)}{(f^n)'(v_j)}=
\sigma \sum_{j\in S}a_j L(c_j, \sigma)=0.
\end{equation}
Thus $A=B=0$, hence, the regularization $\hat H$ of $H$ takes the form:
$$\hat H(z)=H(z)+\frac{A}{z}+\frac{B}{z^2}=H(z),$$
i.e., $H$ is an integrable (on the plane)
fixed point of $T$. By Lemma~\ref{contr} (1), either $H(z)=0$
for every $z$ outside of the set 
$K$, or $f$ is an exceptional rational function. The latter is
excluded, hence, the former holds. By Lemma~\ref{ccomp}
and the definition of $H$, 
we get that all $\alpha_k$ in~(\ref{hformrat}) are zeros. 
In other words, $\sum_{j\in S} a_j H_j=0$.
Now we finish the proof as in the proof of Theorem~\ref{one} (b),
Subsection~\ref{oneb}. 
Note that Corollary~\ref{lemravnie} remains true
for rational functions
(this follows directly from~(\ref{limrata}) and~(\ref{limratinf})
of Theorem~\ref{onerata}).
It allows us to repeat the proof of Lemma~\ref{stranno}
(with obvious changes)
to conclude that all $a_j$, $j\in S$, are zeros.
On the other hand, at least one $a_j$ is not zero, which is
a contradiction.

Remaining cases are quite similar.

($NN_\infty$). 
The relations~(\ref{rlm}) hold. Since $\sigma=1$, then 
the first one gives us $L=0$, and since $b\not=0$,
the second relation gives $A=0$. Besides, (\ref{b}) also holds.
Then we apply Lemma~\ref{contr} (2) and end the proof as in the first case.

($ND_\infty$), i.e. $\sigma=1$ and $b=0$.
Now, assuming the contrary, we get a non-trivial
linear combination $H=\sum_{j\in S} a_j H_j$, such that 
\begin{equation}\label{rln2}
H(z)-(T_fH)(z)=\frac{L_{p-1}}{v_{p-1}-z}+\frac{L_{p}}{v_{p}-z}.
\end{equation}
Then 
$$\hat H(z)-(T_f\hat H)(z)=\frac{A}{z}-
A(\frac{1}{\sigma z}+\frac{b}{\sigma z^2})
-\frac{L_{p-1}+L_p}{z}-\frac{L_{p-1}v_{p-1}+L_p v_p}{z^2}+O(\frac{1}{z^3}),$$
and since the function $\hat H$ is integrable at infinity,
this implies that
\begin{equation}\label{rln3}
A(1-\frac{1}{\sigma})-(L_{p-1}+L_p)=0, \ \ \ \ 
-\frac{A b}{\sigma}-(L_{p-1}v_{p-1}+L_p v_p)=0.
\end{equation}
In the considered case, 
$\sigma=1$, $b=0$, and $v_{p-1}=1\not=0=v_p$, hence, by~(\ref{rln3}),
$L_{p-1}=L_p=0$. Thus $H$ is a fixed point of $T_f$.
We apply Lemma~\ref{contr} (2)
and end the proof as in the first case.
\begin{com}\label{last}
The main results of the paper - Theorem~\ref{one} and Theorem~\ref{onerat} -
can be extended to also include non-repelling periodic orbits.
Consider the case of rational functions 
(the polynomial case is very similar and 
simpler). Assume that $f\in \Lambda_{d, \bar p'}$ has $r$ summable
critical points $c_j$, for $j\in S$, 
and also $r_a$ non-repelling periodic orbits
$O_j$ (different from infinity), for $j\in S_a$,
where $S$ and $S_a$ are two disjoint subsets of $\{1,2,...,p'\}$. 
We assume that the multiplier $\rho_j$ of each $O_j$
is different from $1$,
and, if $\rho_j=0$, $O_j$ contains only a single and simple critical point.
Assume for simplicity that the orbit of 
each $c_j$, $j\in S$, lies in the plane, and all critical values
of $f$ lie in the plane, too. Then the matrix ${\bf L^M}$ in
Theorem~\ref{onerat} is, in fact, the matrix 
${\bf L}$ introduced in Corollary~\ref{lmtol}.
Let us extend the matrix ${\bf L}$ by the following $r_a$ rows.
In the case ($H_\infty$), extend it by
$$(\frac{\partial \rho_j}{\partial \sigma}, 
\frac{\partial \rho_j}{\partial v_1},..., 
\frac{\partial \rho_j}{\partial v_{p-1}})_{j\in S_a},$$
in the case ($NN_\infty$), by
$$(\frac{\partial \rho_j}{\partial v_1},...,
\frac{\partial \rho_j}{\partial v_{p-1}})_{j\in S_a},$$
and in the case ($ND_\infty$), by
$$(\frac{\partial \rho_j}{\partial v_1},...,
\frac{\partial \rho_j}{\partial v_{p-2}})_{j\in S_a}.$$
Then the rank of the extended matrix is maximal, i.e.,
equal to $r+r_a$.

In the proof, we consider the system of $r+r_a$ relations as follows.
For every $j\in S$, we have a relation of Proposition~\ref{vaznoerat}
of the present paper:
\begin{equation}\label{last1}
H_j(x)-(T_f H_j)(x)=\sum_{k=1}^{p}\frac{L(c_j, v_k)}{v_k-x},
\end{equation}
and for every $j\in S_a$, we have a similar relation, see
Theorem 5 of~\cite{mult}:
\begin{equation}\label{last2}
B_j(x)-(T_f B_j)(x)=\sum_{k=1}^{p}\frac{\partial \rho_j}{\partial
v_k}\frac{1}{z-v_k},
\end{equation}
where 
$B_j(z)=\sum_{b\in O_j} (\frac{\rho_j}{(z-b)^2}+\frac{1}{1-\rho_j}
\frac{(f^n)''(b)}{z-b})$.
Then we proceed precisely as 
in the proof of Theorem 6 of~\cite{mult} (or
the proof of Theorem~\ref{onerat} (b) of the present paper).
\end{com}

\end{document}